\documentclass[11pt,a4paper]{amsart}             
    
\usepackage{graphicx}
\usepackage{amssymb}

\usepackage[english]{babel}
\usepackage{enumerate}

\newtheorem{lemma}{Lemma}[section]
\newtheorem{definition}[lemma]{Definition}
\newtheorem{remark}[lemma]{Remark}
\newtheorem{corollary}[lemma]{Corollary}
\newtheorem{theorem}[lemma]{Theorem}

\title[Global convergence of Newton's method]{GLOBAL CONVERGENCE OF NEWTON'S METHOD FOR THE REGULARIZED $p$-STOKES EQUATIONS}
\usepackage{geometry}
\begin{document}
	
	\maketitle
	\begin{center}
	\author{Niko Schmidt\footnote{Kiel University, n\_f\_schmidt@yahoo.de}}
	\end{center}	
	\section*{Abstract}
The motion of glaciers can be simulated with the $p$-Stokes equations. Up to now, Newton's method to solve these equations has been analyzed in finite-dimensional settings only. We analyze the problem in infinite dimensions to gain a new viewpoint.
We do that by proving global convergence of the infinite-dimensional Newton's method with Armijo step sizes to the solution of these equations. 
We only have to add an arbitrarily small diffusion term for this convergence result. We prove that the additional diffusion term only causes minor differences in the solution compared to the original $p$-Stokes equations under the assumption of some regularity. Finally, we test our algorithms on two experiments: A reformulation of the experiment ISMIP-HOM $B$ without sliding and a block with sliding. For the former, the approximation of exact step sizes for the Picard iteration and exact step sizes and Armijo step sizes for Newton's method are superior in the experiment compared to the Picard iteration. For the latter experiment, Newton's method with Armijo step sizes needs many iterations until it converges fast to the solution. Thus, Newton's method with approximately exact step sizes is better than Armijo step sizes in this experiment.
	\subsection*{Keywords}
		$p$-Stokes, Newton's method, global convergence, glaciology, sliding
	\section{Introduction}\label{sec1}
	
	Ice-sheet models spend most of the computation time solving the momentum equations, \cite{Fischler2022}. These equations are nonlinear partial differential equations named $p$-Stokes equations. 
	
	We prove that Newton's method with Armijo step sizes converges to the solution of the $p$-Stokes equations if we add a small diffusion term. 
	We control the step size with a convex functional, which is the anti-derivative of the $p$-Stokes equations. Evaluating this functional has nearly no computational cost. Moreover, the convexity of the functional allows us to approximate the exact step sizes. 
	We slightly modify the numerical experiment ISMIP-HOM $B$ to test Newton's method with Armijo step sizes. Moreover, we test this algorithm with a sliding block. We test both experiments with the Picard iteration and Newton's method with approximations of exact step sizes.
	
	The small diffusion term is necessary to imply G\^{a}teaux differentiability of the $p$-Stokes equations in the variational formulation. 
	Furthermore, the shear-thinning viscosity term has a negative exponent. Thus, we need for differentiability a positive constant in this term. To conclude the theory, we show that the regularized solution converges to the solution of the $p$-Stokes equations under a slight regularity assumption.
	
	Regarding the well-posedness of the equations, there are different types of literature results: nonlinear friction boundary conditions \cite{Chen2013}, an implicitly given viscosity \cite{Jouvet2011} with a differentiable shear-thinning term, or a differentiable shear-thinning term with Dirichlet boundary conditions, \cite{Hirn2013}. One recent publication with more general boundary conditions than we consider in two dimensions uses Newton's method in a finite-dimensional setting, see \cite{Diego2023}. However, we need a combination of a differentiable, explicitly given shear-thinning viscosity, and nonlinear friction boundary conditions. Glaciologists use such a formulation to simulate glaciers, \cite{Fowler2021}. 
	G\^{a}teaux differentiability results are mainly devoted to optimal control and the necessary control-to-state mapping, see for example, \cite{Casas1993} for the $p$-Laplace equations or \cite{Arada2012a} for the Navier-Stokes equations with a nonlinear $p$-Stokes term.
	In \cite{Arada2012a}, the G\^{a}teaux differentiability is shown for $p\geq 2$. The idea of how to prove differentiability for $1<p<2$ is briefly mentioned in \cite[section $6$]{Arada2012}. That paper uses the additional diffusion term for optimal control. It proves convergence in Sobolev spaces for vanishing regularization of the diffusion term. However, that paper has different requirements: It considers optimal control in a slightly different formulation, has Dirichlet boundary conditions, and has a convective term. It has more restrictions on the exponent of the shear-thinning fluids than we have. In \cite{Jouvet2011}, local quadratic convergence of Newton's method is shown for the finite-dimensional case. Glaciologists already consider Newton's method, \cite{Fraters2019}, but we consider a different approach by adding a small diffusive term and using a convex functional for step size control. There is also a multigrid approach with Newton's method and a reformulation of the partial differential equation with first derivatives available, \cite{Allen2017}. 
	
	The paper is structured as follows: In section $\ref{section2}$, we derive the variational formulation, project to divergence-free spaces, introduce a minimization problem that is equivalent to solving the full-Stokes equations, and verify existence and uniqueness of a solution. In section $\ref{DiffG}$, we consider G\^{a}teaux differentiability.
	In section $\ref{NewtonMethod}$, we verify that Newton steps can be calculated. 
	In section \ref{GlobalSuperLinear}, we use the functional to verify global convergence.
	Additionally, we prove that the solution of the regularized $p$-Stokes equations converges to the solution of the $p$-Stokes equations for vanishing regularization under some regularity assumptions. In section $\ref{Experiments}$, we consider one experiment without sliding and one with. We summarize our results in the final section $\ref{Conclusion}$.
	\section{The p-Stokes equations}\label{section2}
	In this section, we formulate the $p$-Stokes equations in both the classical and the variational formulations. The $p$-Stokes equations can be used to, e.g., model the motion of glaciers. Their complexity results from nonlinear viscosity, also described as shear-thinning, \cite{Fowler2021}. Mass conservation and incompressibility lead to a divergence-free velocity $\boldsymbol{v}$. Let $N \in \{2,3\}$, $\Omega\subseteq\mathbb{R}^N$ be a Lipschitz domain. Let $\sigma$ be the stress tensor, $\boldsymbol{v}$ the velocity, $\rho$ the density, and $\boldsymbol{g}$ the gravitational acceleration. The $p$-Stokes equations are:
	\begin{align}
	-\mathrm{div}\sigma &=-\rho \boldsymbol{g}\quad \text{on }\Omega,\label{FirstPart}
	\\
	\mathrm{div}\boldsymbol{v}&=0\quad \quad \text{ on }\Omega
	\end{align}
	To relate the stress tensor $\sigma$, the velocity $\boldsymbol{v}$, and the pressure $\pi$, we introduce the identity matrix $I=(I_{ij})_{ij}$,
	\begin{align*}
	(D\boldsymbol{v})_{ij}&=\frac{1}{2}\bigg(\frac{\partial v_i}{\partial x_j}+\frac{\partial v_j}{\partial x_i}\bigg)\text{, }\quad i,j\in \{1,...,N\}, \\
	B&\in L^{\infty}(\Omega)\text{, }\quad B\geq c_0\in (0,\infty),
	\end{align*}
	and the following definition:
	\begin{definition}
		Let $p\in (1,2)$, $\delta>0$. We define $S^p:\mathbb{R}^{N \times N}\to \mathbb{R}^{N \times N}$,
		\begin{align}
		S^p(P)=\big(|P|^2+\delta^2\big)^{(p-2)/2}P.
		\end{align}
		The norm $|\cdot |$ is the Frobenius norm. For vectors $\boldsymbol{v}\in \mathbb{R}^n$, we interpret $S^p$ as a mapping from $\mathbb{R}^n\to \mathbb{R}^n$ with
		\begin{align*}
		S^p(\boldsymbol{v})=\big(|\boldsymbol{v}|^2+\delta^2\big)^{(p-2)/2}\boldsymbol{v}
		\end{align*}
		with the Euclidean norm $|\cdot |$.
	\end{definition}
	Let $p\in (1,2)$. The stress tensor $\sigma$ is given by
	\begin{align*}
	\sigma=-\pi I+BS^p(D\boldsymbol{v}).
	\end{align*}
	The case $p=2$ reduces to the Stokes problem. In glaciological applications $p=4/3$, see \cite[section 1.4.3]{Fowler2021}, or in more recent approaches $p=5/4$, see \cite[Abstract]{Bons2018}, is used. The following result states information about the integrability of $S^p$:
	\begin{lemma}\label{LqS}
		Let $p\in (1,2)$. For all $P \in L^p(\Omega)^{N \times N}$ follows $S^p(P)\in L^{p'}(\Omega)^{N \times N}$ with the dual exponent $p'$.
	\end{lemma}
	\begin{proof}
		With the dual exponent $p'=p/(p-1)$ and $p\in (1,2)$ follows
		\begin{align*}
		|S^p(P)|\leq |P|^{p-1}\Rightarrow |S^p(P)|^{p'} \leq |P|^p\in L^1(\Omega).
		\end{align*}
	\end{proof}
	Let $\partial \Omega = \Gamma_d \cup \Gamma_a \cup \Gamma_b$. The Dirichlet boundary condition $\boldsymbol{v}=0$ is satisfied for those parts of the glacier frozen to the ground $\Gamma_d$. The interaction of the glacier with the air is given by $\sigma \cdot \boldsymbol{n}=0$ on $\Gamma_a$ with
	the matrix-vector multiplication of the matrix $\sigma$ and the unit normal vector $\boldsymbol{n}$. We assume nonlinear sliding at parts of the bedrock $\Gamma_b$ that are not frozen to the ground. This sliding is represented by tangential sliding. For these parts of the bedrock, we assume that the normal component of the velocity is $0$ because we neglect the melting of ice or freezing of water at the bedrock. We remind of the definition of the tangential components
	\begin{align}\label{Tangential}
	\boldsymbol{v_t}=\boldsymbol{v}-(\boldsymbol{v}\cdot \boldsymbol{n})\boldsymbol{n}\quad\text{ and }\quad
	\boldsymbol{\sigma_t}=\sigma\cdot \boldsymbol{n}-(\sigma \cdot \boldsymbol{n}\cdot \boldsymbol{n})\boldsymbol{n}.
	\end{align}
	Let $\tau\in L^{\infty}(\Gamma_b)$, $\tau\geq 0$, $s\in (1,p]$. We formulate the boundary conditions as follows:
	\begin{align}
	\boldsymbol{v}&=0 \quad\quad\quad\quad\quad\text{on }\Gamma_d, \label{Gammad} \\
	\sigma\cdot \boldsymbol{n} &= 0 \quad\quad\quad\quad\quad\text{on }\Gamma_a \label{Gammaa}
	\\  
	\boldsymbol{\sigma}_t&=-\tau S^s(\boldsymbol{v_t})\quad \text{ on } \Gamma_b\text{,}\label{Gammab1}
	\\
	\boldsymbol{v}\cdot \boldsymbol{n}&=0\quad\quad\quad\quad\quad \text{on }\Gamma_b.\label{Gammab}
	\end{align}
	For the choice of the boundary conditions see also \cite{Chen2013}. The boundary conditions imply as a natural choice of the space
	\begin{definition}[Function space]\label{Wp}
		Let $|\Gamma_d|>0$. We define for all $p>1$
		\begin{align}
		W_p:=\{\boldsymbol{v}\in W^{1,p}(\Omega)^N\text{; }\boldsymbol{v}|_{\Gamma_d}=0\text{, }\boldsymbol{v}|_{\Gamma_b}\cdot \boldsymbol{n}|_{\Gamma_b}=0\}
		\end{align}
		with norm
		\begin{align*}
		\|\boldsymbol{v}\|_{W_p}^p=\int_{\Omega}\bigg(\sum_{i,j=1}^N\Big(\frac{\partial v_i}{\partial x_j}\Big)^2\bigg)^{p/2}\, dx.
		\end{align*}
	\end{definition}
	The restriction $|\Gamma_d|>0$ is also fulfilled in common applications, see for example \cite[Abstract]{Bons2018} or \cite[Abstract]{MacGregor2016}. 
	
	The Poincar\'{e} inequality implies that the $W_p$-norm is equivalent to the $W^{1,p}(\Omega)^N$-norm. For example, the Poincar\'{e} inequality was proved in \cite[Theorem 1.5]{Chipot2002} for $p=2$. But the proof is identical for $p\in (1,2)$.
	\subsection{Variational formulation}\label{MonotonOpertors}
	We derive the variational formulation in the space $W_p$. We define the double scalar product and obtain the following equation for $\sigma,\tau \in \mathbb{R}^{N \times N}$ immediately:
	\begin{align}\label{Transpose}
	\sigma:\tau:=\sum_{i,j=1}^N \sigma_{ij}\tau_{ij}\text{, }\quad \sigma:\tau^T=\sum_{i,j=1}^N\sigma_{ij}\tau_{ji}=\sigma^T:\tau.
	\end{align}
	Let $p'$ be the dual exponent, $\boldsymbol{\phi}\in W_{p'}$, $\sigma \in W^{1,p}(\Omega)^{N \times N}$. We conclude by multiplication with the test function $\boldsymbol{\phi}$ and partial integration for the left-hand side of equation $(\ref{FirstPart})$
	\begin{align}\label{VarSummands}
	-\int_{\Omega}\mathrm{div}\sigma \cdot \boldsymbol{\phi}\, dx= \int_{\Omega} \sigma :\nabla \boldsymbol{\phi}\, dx-\int_{\partial \Omega}\sigma\cdot \boldsymbol{n}\cdot \boldsymbol{\phi}\, ds.
	\end{align}
	Let $\boldsymbol{v}\in W_{p}$, $\pi \in L^{p'}(\Omega)$. We use $\sigma=-\pi I+BS^p(D\boldsymbol{v})$ and $I:\nabla \boldsymbol{\phi}=\mathrm{div}\boldsymbol{\phi}$ to obtain for the first summand on the right-hand side of equation $(\ref{VarSummands})$
	\begin{align*}
	\int_{\Omega}\sigma :\nabla \boldsymbol{\phi} \, dx
	=-\int_{\Omega}\pi \mathrm{div}\boldsymbol{\phi}\, dx +\int_{\Omega}BS^p(D\boldsymbol{v}):\nabla \boldsymbol{\phi} \, dx.
	\end{align*}
	The second summand on the right-hand side of equation $(\ref{VarSummands})$ vanishes on $\Gamma_d$ because the test function $\boldsymbol{\phi}\in W_{p'}$ vanishes on $\Gamma_d$, see the definition of $W_{p'}$ in Definition $\ref{Wp}$. Moreover, $\sigma\cdot \boldsymbol{n}=0$ is valid on $\Gamma_a$, see equation $(\ref{Gammaa})$. Thus, the integral over this domain disappears, too. It follows
	\begin{align*}
	-\int_{\partial \Omega}\sigma\cdot \boldsymbol{n}\cdot \boldsymbol{\phi}\, ds
	=-\int_{\Gamma_b}\sigma\cdot \boldsymbol{n}\cdot \boldsymbol{\phi}\, ds.
	\end{align*}
	The relation between tangential and normal components, see $(\ref{Tangential})$, implies
	\begin{align*}
	(\sigma \cdot \boldsymbol{n})\cdot \boldsymbol{\phi}=(\sigma \cdot \boldsymbol{n}\cdot \boldsymbol{n})(\boldsymbol{\phi}\cdot \boldsymbol{n})+\boldsymbol{\sigma_t}\cdot \boldsymbol{\phi_t}.
	\end{align*}
	We split the boundary integral on $\Gamma_b$ into normal and tangential components and use the boundary conditions on $\Gamma_b$, see equation $(\ref{Gammab1})$, and the definition of $\boldsymbol{\phi}\in W_{p'}$, namely $\boldsymbol{\phi}\cdot \boldsymbol{n}=0$ on $\Gamma_b$, see Definition $\ref{Wp}$, to obtain
	\begin{align*}
	-\int_{\Gamma_b}\sigma \cdot \boldsymbol{n} \cdot \boldsymbol{\phi}\, ds
	&=-\int_{\Gamma_b}(\sigma \cdot \boldsymbol{n} \cdot \boldsymbol{n}) (\boldsymbol{\phi} \cdot \boldsymbol{n})\, ds-\int_{\Gamma_b}\boldsymbol{\sigma}_t \cdot \boldsymbol{\phi}_t \, ds
	\\
	&=\int_{\Gamma_b}\tau \big(|\boldsymbol{v_t}|^2+\delta^2\big)^{(s-2)/2}\boldsymbol{v_t}\cdot \boldsymbol{\phi}_t \, ds.
	\end{align*}
	Due to equation $(\ref{Tangential})$ and equation $(\ref{Gammab})$, we have $\boldsymbol{v}=\boldsymbol{v_t}$ on $\Gamma_b$, which yields
	\begin{align*}
	\int_{\Gamma_b}\tau \big(|\boldsymbol{v_t}|^2+\delta^2\big)^{(s-2)/2}\boldsymbol{v_t}\cdot \boldsymbol{\phi}_t \, ds
	&=\int_{\Gamma_b}\tau \big(|\boldsymbol{v}|^2+\delta^2\big)^{(s-2)/2}\boldsymbol{v}\cdot \boldsymbol{\phi}\, ds
	\\
	&=\int_{\Gamma_b}\tau S^s(\boldsymbol{v})\cdot \boldsymbol{\phi} \, ds.
	\end{align*}
	We set $L^{p'}_0(\Omega):=\{\pi \in L^{p'}(\Omega)\text{; } \int_{\Omega}\pi \, dx=0\}$. Then a weak solution of the problem is given by $(\boldsymbol{v},\pi)\in (W_p,L^{p'}_0(\Omega))$ with
	\begin{align}\label{MixedElementProblem}
	\begin{split}
	\langle A\boldsymbol{v},\boldsymbol{\phi}\rangle_{W_p^*,W_p}-(\pi,\mathrm{div}\boldsymbol{\phi})&=(\rho \boldsymbol{g},\boldsymbol{\phi})\quad \text{for all }\boldsymbol{\phi}\in W_p,
	\\
	-(\mathrm{div}\boldsymbol{v},\psi)&=0\qquad \qquad\text{ for all }\psi \in L^{p'}_0(\Omega)
	\end{split}
	\end{align}
	with the operator $A:W_p\to W_p^*$,
	\begin{align*}
	\langle A\boldsymbol{v},\boldsymbol{\phi}\rangle_{W_p^*,W_p}
	=\int_{\Omega}BS^p(D\boldsymbol{v}):\nabla \boldsymbol{\phi}\, dx+\int_{\Gamma_b}\tau S^s(\boldsymbol{v})\cdot \boldsymbol{\phi}\, ds\text{, }\quad \boldsymbol{\phi}\in W_p.
	\end{align*}
	The operator $A$ is well-defined due to Lemma $\ref{LqS}$.
	\begin{definition}[Divergence free space]
		Let $|\Gamma_d|>0$. We define for all $p>1$ the divergence-free velocity space 
		\begin{align}
		V_p:=\{\boldsymbol{v}\in W_p\text{; }\mathrm{div}\boldsymbol{v}=0\}
		\end{align}
		with $W_p$ as in Definition $(\ref{Wp})$.
	\end{definition}
	Now, we examine surjectivity of the restricted operator $A:V_p\to V_p^*$,
	\begin{align}\label{WeakSolution}
	\langle A\boldsymbol{v},\boldsymbol{\phi}\rangle_{V_p^*,V_p}=\int_{\Omega}\rho \boldsymbol{g}\cdot \boldsymbol{\phi}\, dx \quad \text{ for all }\phi \in V_p.
	\end{align}
	By solving the divergence-free problem, we solve the original problem: The inf-sup condition is proved by \cite[Corollary 3.2]{Amrouche1994} for Dirichlet boundary conditions:
	\begin{align*}
	\inf_{\pi \in L^{p'}_0(\Omega)}\sup_{\boldsymbol{v}\in W^{1,p}_0(\Omega)^N}\frac{(\pi,\mathrm{div}\boldsymbol{v})}{\|\boldsymbol{v}\|_{W^{1,p}_0(\Omega)^N}\|\pi\|_{L^{p'}(\Omega)}}\geq c>0.
	\end{align*}
	The inf-sup condition is also fulfilled for $W_p\supseteq W^{1,p}_0(\Omega)^N$. Thus, we can apply \cite[Theorem IV.1.4]{Girault1986}. This yields for each solution in the divergence-free formulation a unique solution for the original problem. This theorem is stated in Hilbert spaces, but the proof is identical for operators on Banach spaces with an existing dual operator.
	\subsection{Regularization of the equation and the divergence-free space}
	In this subsection, we introduce a regularization and the divergence-free space. We add a small diffusion term that allows us to obtain a solution in $V_2$.  We need solutions in $V_2$ and not only in $V_p$ to obtain G\^{a}teaux differentiability in section $\ref{DiffG}$. We define for $\mu_0>0$ the operator $A:V_2\to V_2^*$,
	\begin{align}\label{OperatorA}
	\langle A\boldsymbol{v},\boldsymbol{\phi}\rangle_{V_2^*,V_2}=\langle A\boldsymbol{v},\boldsymbol{\phi}\rangle_{V_p^*,V_p}+\mu_0(\nabla \boldsymbol{v},\nabla \boldsymbol{\phi}).
	\end{align}
	
	We define a solution of the $p$-Stokes equations:
	\begin{definition}[Weak solution of the $p$-Stokes equations]\label{WeakSolutionDef}
		Let $p\in (1,2)$, $|\Gamma_d|>0$, $\delta>0$, and $\mu_0>0$. A solution of 	
		\begin{align}\label{WeakSolutionNew}
		\langle A\boldsymbol{v},\boldsymbol{\phi}\rangle_{V_2^*,V_2}=\int_{\Omega}\rho \boldsymbol{g}\cdot \boldsymbol{\phi}\, dx\quad \text{ for all }\boldsymbol{\phi} \in V_2.
		\end{align}
		is called weak solution of the $p$-Stokes equations.
	\end{definition}
	For $\mu_0=0$ and $\delta=0$ existence and uniqueness of a solution to equation $(\ref{WeakSolution})$ is shown in \cite{Chen2013}; for $\mu=0$, $\delta>0$, and $\Gamma_d=\partial \Omega$, see \cite{Hirn2013}. We consider $\mu_0>0$ and $\delta>0$ because we want G\^{a}teaux differentiability in the infinite-dimensional space. In finite dimensions, $\mu_0>0$ is not necessary, see \cite[section $3$]{Hirn2013}. 
	
	For vanishing $\mu_0$ and $\delta$, the weak solution of the $p$-Stokes equations converges to the solution of $\mu_0=0$ and $\delta=0$ under slight regularity assumptions, see section $\ref{ConvergenceRegularisation}$. 
	\subsection{Equivalent minimization problem}\label{EquivalentProblem}
	In this subsection, we introduce a minimization problem, which is equivalent to finding a solution to the regularized $p$-Stokes equations. In \cite{Hirn2013}, the convex functional was introduced for $\mu_0=0$ and Dirichlet boundary conditions. In \cite{Chen2013}, it was introduced for $\mu_0=0$ and $\delta=0$ for more general boundary conditions. We use those formulations for our convex functional with $\mu_0,\delta \in [0,\infty)$.
	\begin{definition}[Convex functional]\label{ConvexFunction}
		Let $\mu_0,\delta\in [0,\infty)$, $J_{\mu_0,\delta}:V_r\to \mathbb{R}$ with $r=2$ for $\mu_0>0$ and $r=p$ for $\mu_0=0$,
		\begin{align}
		J_{\mu_0,\delta}(\boldsymbol{v})=&\int_{\Omega}\frac{B}{p}\big(|D\boldsymbol{v}|^2+\delta^2\big)^{p/2}\, dx
		+\int_{\Gamma_b}\frac{\tau}{s}\big(|\boldsymbol{v}|^2+\delta^2\big)^{s/2}\, ds+\mu_0(\nabla \boldsymbol{v},\nabla \boldsymbol{v})-(\rho \boldsymbol{g},\boldsymbol{v}).
		\end{align}
	\end{definition}
	
	\begin{lemma}
		Let $\mu_0,\delta\in [0,\infty)$. The functional $J_{\mu_0,\delta}$ is Fr\'{e}chet differentiable with
		\begin{align*}
		J_{\mu_0,\delta}'(\boldsymbol{v})\boldsymbol{w}=\langle A\boldsymbol{v},\boldsymbol{w}\rangle_{V_2^*,V_2}-(\rho \boldsymbol{g},\boldsymbol{w}).
		\end{align*}
	\end{lemma}
	\begin{proof}
		For $\delta>0$,  the G\^{a}teaux differentiability of the first and fourth summand are discussed in \cite{Hirn2013} on $W^{1,p}_0(\Omega)^N$. Because the boundary conditions do not influence this integral, the G\^{a}teaux differentiability is also clear on $V_2$. The second summand can be handled identically to the first summand. Also the term $\boldsymbol{v}\mapsto \mu_0(\nabla \boldsymbol{v},\nabla \boldsymbol{v})$ is G\^{a}teaux differentiable. 
		
		For $(\mu_0,\delta)=0$, differentiability is proved in \cite{Chen2013}. Adding the summand $\mu_0(\nabla \boldsymbol{v},\nabla \boldsymbol{v})$ does not change the differentiability result.
		
		Thus, $J_{\mu_0,\delta}$ is G\^{a}teaux differentiable. Because $A$ and $(\boldsymbol{v},\boldsymbol{w})\mapsto (\rho \boldsymbol{g},\boldsymbol{w})$ are continuous, $J_{\mu_0,\delta}$ is Fr\'{e}chet differentiable.
	\end{proof}
	\subsection{Existence and uniqueness of weak solutions in the divergence-free space}
	Differentiability and uniqueness were proven for $\delta=0=\mu_0$ in \cite{Chen2013} by using the strictly convex functionals as in Definition $\ref{ConvexFunction}$. For Dirichlet zero boundary conditions existence and uniqueness of solutions were stated without detailed proof via the Browder-Minty Theorem, e.g., in \cite{Berselli2017}. For completeness, we prove existence and uniqueness of a solution with the Browder-Minty Theorem as our boundary conditions, see Equations $(\ref{Gammad})$ to $(\ref{Gammab})$ are more complicated than the standard Dirichlet boundary conditions and with $\delta>0$ and $\mu_0>0$ different to the problem in \cite{Chen2013}. Thus, we prove a well-known result for a slightly different formulation.
	
	Before we can state the existence and uniqueness result, we have to introduce the following definitions:
	\begin{definition}[Strict monotonicity, coercivity]
		Let $X$ be a reflexive Banach space, $A:X\to X^*$. The operator $A$ is called strictly monotone, if we have for all $\boldsymbol{v},\boldsymbol{w}\in X$
		\begin{align*}
		\langle A\boldsymbol{v}-A\boldsymbol{w},\boldsymbol{v}-\boldsymbol{w}\rangle_{X^*,X}>0.
		\end{align*}
		The operator $A$ is called coercive, if we have
		\begin{align*}
		\lim_{\|\boldsymbol{v}\|_X\to \infty}\frac{\langle A\boldsymbol{v},\boldsymbol{v}\rangle_{X^*,X}}{\|\boldsymbol{v}\|_X}=\infty.
		\end{align*}
	\end{definition}
	The conditions for existence and uniqueness are formulated in the following theorem.
	\begin{theorem}[Browder-Minty Theorem]\label{BrowderMinty}
		Let $X$ be a reflexive separable Banach space and $A:X\to X^*$ a strictly monotone, coercive, and continuous operator. Let $f\in X^*$. Then, there exists a unique solution $u\in X$ for the equation 
		\begin{align*}
		Au=f\quad \text{ in }X^*.
		\end{align*}
	\end{theorem}
	\begin{proof} The existence of $u\in X$ is proved in \cite{Browder1963} Theorem $2$ in a more general version. Uniqueness follows immediately with the strict monotonicity of $A$.
	\end{proof}
	We analyze $S^p$ and $S^s$ to verify that $A$ is strictly monotone, coercive, and continuous. For that purpose, we need the following result:
	\begin{lemma}\label{PDeltaConditions}
		Let $r\in (1,2)$, $S:\mathbb{R}^{N \times N}\to \mathbb{R}^{N \times N}$ with
		\begin{align}
		\sum_{i,j,k,\ell=1}^N\frac{\partial S_{ij}(P)}{\partial P_{kl}}Q_{ij}Q_{k\ell}&\geq c_1(|P|+\delta)^{r-2}|Q|^2,\label{LowerBound}
		\\
		\bigg|\frac{\partial S_{ij}(P)}{\partial P_{k\ell}}\bigg|&\leq c_2(|P|+\delta)^{r-2}\quad \text{ for all }i,j,k,\ell\in \{1,...,N\}\label{UpperBound}
		\end{align}
		for all $P,Q\in \mathbb{R}^{N \times N}$, $c_1,c_2\in (0,\infty)$, $\delta\geq 0$. Then, there exist $c,C\in \mathbb{R}$ independent of $\delta$ for all $P,Q\in \mathbb{R}^{N \times N}$ with
		\begin{align*}
		(S(P)-S(Q)):(P-Q)&\geq c (\delta+|P|+|Q|)^{r-2}|P-Q|^2,
		\\
		|S(P)-S(Q)|&\leq  C (\delta+|P|+|Q|)^{r-2}|P-Q|.
		\end{align*}
	\end{lemma}
	\begin{proof}
		See \cite[Lemma 6.3]{Diening2007}.
	\end{proof}
	We calculate the derivative of $S^r$, $r\in (1,2)$ and verify the properties stated above to apply the Browder-Minty Theorem $\ref{BrowderMinty}$:
	\begin{lemma}\label{LipschitzEstimate}
		Let $r\in (1,2)$. The function $S^r_{kl}$ is continuously differentiable. Furthermore, the inequalities $(\ref{LowerBound})$ and $(\ref{UpperBound})$ are fulfilled for $S^r$.
	\end{lemma}
	\begin{proof}
		We verify the conditions in Lemma $\ref{PDeltaConditions}$. Let $P\in \mathbb{R}^{N \times N}$, $i,j\in \{1,...,N\}$. We calculate
		\begin{align*}
		\frac{\partial}{\partial P_{ij}}\big(|P|^2+\delta^2\big)^{(r-2)/2}=(r-2)\Bigg(\bigg(\sum_{k,\ell=1}^NP_{k\ell}^2\bigg)+\delta^2\Bigg)^{(r-4)/2}P_{ij}.
		\end{align*}
		Let $k,\ell\in \{1,..,N\}$. We infer
		\begin{align}\label{DerivativeSr}
		\begin{split}
		\frac{\partial S^r_{k\ell}}{\partial P_{ij}}(P)&=\frac{\partial}{\partial P_{ij}}\Big(\big(|P|^2+\delta^2\big)^{(r-2)/2}P_{k\ell}\Big)
		\\
		&=(r-2)\big(|P|^2+\delta^2\big)^{(r-4)/2}P_{ij}P_{k\ell}+\big(|P|^2+\delta^2\big)^{(r-2)/2}I_{ik}I_{j\ell}.
		\end{split}
		\end{align}
		We can verify inequality $(\ref{UpperBound})$ with $\tilde{c}_1\in \mathbb{R}$ now:
		\begin{align*}
		\bigg|\frac{\partial S^r_{k\ell}(P)}{\partial P_{ij}}\bigg|&\leq N^2\Big((2-r)\big(|P|^2+\delta^2\big)^{(r-4)/2}|P|^2+\big(|P|^2+\delta^2\big)^{(r-2)/2}\Big)
		\\
		&= N^2\Big((2-r)\big(|P|^2+\delta^2\big)^{(r-2)/2}\frac{|P| \, |P|}{|P|^2+\delta^2}+\big(|P|^2+\delta^2\big)^{(r-2)/2}\Big)
		\\
		&\leq \tilde{c}_1N^2(3-r)(|P|+\delta)^{r-2}.
		\end{align*}
		We verify inequality $(\ref{LowerBound})$ next. Let $Q\in \mathbb{R}^{N \times N}$. We conclude
		\begin{align*}
		&\quad\sum_{i,j,k,\ell=1}^N\frac{\partial}{\partial P_{ij}}\big(S^r_{k\ell}(P)\big)Q_{k\ell}Q_{ij}
		\\
		&=\big(|P|^2+\delta^2\big)^{(r-2)/2}\bigg(\sum_{i,j,k,\ell=1}^N(r-2)\frac{P_{ij}P_{k\ell}}{|P|^2+\delta^2}Q_{k\ell}Q_{ij}
		+\sum_{i,j,k,\ell=1}^NI_{ik}I_{jl}Q_{k\ell}Q_{ij}\bigg)
		\\
		&=\big(|P|^2+\delta^2\big)^{(r-2)/2}\bigg((r-2)\frac{1}{|P|^2+\delta^2}\sum_{k,\ell=1}^NP_{k\ell}Q_{k\ell}\sum_{i,j=1}^NP_{ij}Q_{ij}+\sum_{i,j=1}^NQ_{ij}^2\bigg)
		\\
		&=\big(|P|^2+\delta^2\big)^{(r-2)/2}\bigg((r-2)\frac{(P:Q)^2}{|P|^2+\delta^2}+|Q|^2\bigg).
		\end{align*}
		We conclude with $r-2\leq 0$, $(P:Q)^2\leq |P|^2|Q|^2$, and $\tilde{c}_2\in \mathbb{R}$
		\begin{align*}
		\big(|P|^2+\delta^2\big)^{(r-2)/2}\bigg((r-2)\frac{(P:Q)^2}{|P|^2+\delta^2}+|Q|^2\bigg)
		&\geq  \big(|P|^2+\delta^2\big)^{(r-2)/2}(r-1)|Q|^2
		\\
		&\geq  \tilde{c}_2(r-1)(|P|+\delta)^{r-2}|Q|^2.
		\end{align*}
		We set $P:=I_{ij}p$ and $Q:=I_{ij}q$ with $p,q\in \mathbb{R}^N$ for the vector-valued situation.
	\end{proof}
	We verify the three properties for the Browder-Minty Theorem, see $\ref{BrowderMinty}$ in the three following lemmata. We verify Lipschitz continuity of $A$ instead of continuity:
	\begin{lemma}\label{Lipschitz}
		Let $|\Gamma_d|>0$, $\delta>0$, $\mu_0>0$, and $p\in(1,2)$. The operator $A:V_2\to V_2^*$, see equation $(\ref{OperatorA})$ is Lipschitz continuous.
	\end{lemma}
	\begin{proof}
		Let $\boldsymbol{v},\boldsymbol{w},\boldsymbol{\phi}\in V_2$. We obtain for the second summand of the operator $A$ with Lemma $\ref{LipschitzEstimate}$ and $C\in \mathbb{R}$
		\begin{align*}
		\bigg|\int_{\Gamma_b}\tau S^s(\boldsymbol{v}-\boldsymbol{w})\cdot \boldsymbol{\phi}\, ds\bigg|
		&\leq 
		C\int_{\Gamma_b}(\delta+|\boldsymbol{v}|+|\boldsymbol{w}|)^{s-2}\, |\boldsymbol{v}-\boldsymbol{w}|\, |\boldsymbol{\phi}|\, ds
		\\
		&\leq
		C\int_{\Gamma_b}\qquad \delta^{s-2}\qquad \qquad |\boldsymbol{v}-\boldsymbol{w}|\, |\boldsymbol{\phi}|\, ds
		\\
		&\leq C\delta^{s-2}\|\boldsymbol{v}-\boldsymbol{w}\|_{L^2(\Gamma_b)}\|\boldsymbol{\phi}\|_{L^2(\Gamma_b)}.
		\end{align*}
		We infer with the trace operator, \cite[Theorem 1.12]{Hinze2009}, and $C_1\in \mathbb{R}$
		\begin{align*}
		\|\boldsymbol{v}-\boldsymbol{w}\|_{L^2(\Gamma_b)}\|\boldsymbol{\phi}\|_{L^2(\Gamma_b)}\leq C_1 \|\boldsymbol{v}-\boldsymbol{w}\|_{V_2}\|\boldsymbol{\phi}\|_{V_2}.
		\end{align*}
		The same arguments are valid for the first summand of the operator $A$ with $\Omega$ instead of $\Gamma_b$, $B$ instead of $\tau$, and $s=p$. Thus, we obtain Lipschitz continuity with $C_2\in \mathbb{R}$ and
		\begin{align*}
		\sup_{\boldsymbol{\phi}\in V_2}|\langle A\boldsymbol{v}-A\boldsymbol{w},\boldsymbol{\phi}\rangle_{V_2^*,V_2}|\leq C_2 \|\boldsymbol{v}-\boldsymbol{w}\|_{V_2}\|\boldsymbol{\phi}\|_{V_2}.
		\end{align*}
	\end{proof}
	\begin{lemma}\label{coerz}
		Let $|\Gamma_d|>0$, $\delta>0$, $\mu_0>0$, and $p\in(1,2)$. The operator $A$ is coercive.
	\end{lemma}
	\begin{proof}
		We know
		\begin{align*}
		\int_{\Omega}B(|D\boldsymbol{v}|^2+\delta^2)^{(p-2)/2}|D\boldsymbol{v}|^2\, dx+\int_{\Gamma_b}\tau(|\boldsymbol{v}|^2+\delta^2)^{(s-2)/2}\boldsymbol{v}\cdot \boldsymbol{v}\, ds \geq 0.
		\end{align*}
		Trivially, we have
		\begin{align*}
		\frac{\mu_0 \|\boldsymbol{v}\|_{V_2}^2}{\|\boldsymbol{v}\|_{V_2}}=\mu_0\|\boldsymbol{v}\|_{V_2}\to \infty \text{ for }\|\boldsymbol{v}\|_{V_2}\to \infty.
		\end{align*}
	\end{proof}
	\begin{lemma}\label{StrictMonotone}
		Let $|\Gamma_d|>0$, $\delta>0$, $\mu_0>0$, and $p\in(1,2)$. The operator $A$ is strictly monotone.
	\end{lemma}
	\begin{proof}
		The strict monotonicity of $S^p$ and $S^s$, proved in Lemma $\ref{LipschitzEstimate}$, yield monotonicity of $A$: Let $\boldsymbol{v},\boldsymbol{w}\in V_2$. We have
		\begin{align*}
		&\langle A\boldsymbol{v}-A\boldsymbol{w},\boldsymbol{v}-\boldsymbol{w}\rangle_{V_2^*,V_2}
		\\
		=&\int_{\Omega}B\big(S^p(D\boldsymbol{v})-S^p(D\boldsymbol{w})\big):(\nabla \boldsymbol{v}-\nabla \boldsymbol{w})\, dx+\int_{\Gamma_b}\tau \big(S^s(\boldsymbol{v})-S^s(\boldsymbol{w})\big)\cdot (\boldsymbol{v}-\boldsymbol{w})\, ds
		\\
		\geq& \int_{\Omega}Bc(\delta+|D\boldsymbol{v}|+|D\boldsymbol{w}|)^{p-2}|D\boldsymbol{v}-D\boldsymbol{w}|^2\, dx+\int_{\Gamma_b}\tau (\delta+|\boldsymbol{v}|+|\boldsymbol{w}|)^{p-2}|\boldsymbol{v}-\boldsymbol{w}|\, ds\geq 0.
		\end{align*}
		The operator $A$ is strict monotone because
		\begin{align*}
		\mu_0 (\boldsymbol{v}-\boldsymbol{w},\boldsymbol{v}-\boldsymbol{w})=\mu_0\|\boldsymbol{v}-\boldsymbol{w}\|_{V_2}^2.
		\end{align*}
	\end{proof}
	Hence, we can prove our existence result:
	\begin{theorem}\label{UniqueSolution}
		Let $|\Gamma_d|>0$, $\delta>0$, $\mu_0>0$, and $p\in(1,2)$. There exists exactly one weak solution of the divergence-free regularized $p$-Stokes equations, see Definition $(\ref{WeakSolutionDef})$.
	\end{theorem}
	\begin{proof}
		(In \cite{Hirn2013} the case $\Gamma_d=\partial \Omega$ is analyzed.) Lemma $\ref{Lipschitz}$ yields Lipschitz continuity of the operator $A$ and thus continuity. We verified the coercivity of $A$ in Lemma $\ref{coerz}$.  The operator $A$ is also strictly monotone, see Lemma $\ref{StrictMonotone}$. We conclude the existence and uniqueness of a solution with the Browder-Minty Theorem, see Theorem $\ref{BrowderMinty}$.
	\end{proof}
	\section{Differentiability}\label{DiffG}
	We calculate the G\^{a}teaux derivative of the operator $A$ to apply Newton's method. At first, we only consider the first summand of the operator $A$. 
	\begin{definition}[Root problem]
		Let $\delta>0$, $\mu_0>0$, $B\in L^{\infty}(\Omega)$, and $\tau \in L^{\infty}(\Gamma_b)$. We define the operator $G:V_2\to V_2^*$ for all $\boldsymbol{\phi} \in V_2$ by
		\begin{align}
		\langle G(\boldsymbol{v}),\boldsymbol{\phi})\rangle_{V_2^*,V_2} &= \langle A\boldsymbol{v},\boldsymbol{\phi}\rangle_{V_2^*,V_2}-(\rho \boldsymbol{g},\boldsymbol{\phi}) \nonumber
		\\
		&=\int_{\Omega}BS^p(D\boldsymbol{v}):\nabla \boldsymbol{\phi}\, dx+\int_{\Gamma_b}\tau S^s(\boldsymbol{v})\cdot \boldsymbol{\phi}\, ds+\mu_0(\nabla \boldsymbol{v},\boldsymbol{\phi})-(\rho \boldsymbol{g},\boldsymbol{\phi}) \label{Gright}
		\end{align}
	\end{definition}
	The operator $G$ is G\^{a}teaux differentiable:
	\begin{theorem}\label{continuity}
		Let $\delta>0$, $\mu_0>0$, $B\in L^{\infty}(\Omega)$, $\tau \in L^{\infty}(\Gamma_b)$, and $\boldsymbol{v},\boldsymbol{w}\in V_2$. The directional derivative has the form
		\begin{align}\label{Derivative}
		\begin{split}
		&\quad\langle G'(\boldsymbol{v})\boldsymbol{w},\boldsymbol{\phi} \rangle_{V_2^*,V_2}
		\\
		&=\int_{\Omega}(p-2)B\big(|D\boldsymbol{v}|^2
		+\delta^2\big)^{(p-4)/2}(D\boldsymbol{v}:D\boldsymbol{w})\, (D\boldsymbol{v}:\nabla \boldsymbol{\phi}) \, dx
		\\
		&\quad+\int_{\Omega}B\big(|D\boldsymbol{v}|^2+\delta^2\big)^{(p-2)/2}D\boldsymbol{w}:\nabla \boldsymbol{\phi} \, dx
		\\
		&\quad+\int_{\Gamma_b}(s-2)\tau \big(|\boldsymbol{v}|^2+\delta^2\big)^{(s-4)/2}(\boldsymbol{v}\cdot \boldsymbol{w})\, (\boldsymbol{v}\cdot \boldsymbol{\phi})\, ds
		\\
		&\quad+\int_{\Gamma_b}\tau \big(|\boldsymbol{v}|^2+\delta^2\big)^{(s-2)/2}\boldsymbol{w}\cdot \boldsymbol{\phi} \, ds
		+\mu_0\int_{\Omega}\nabla \boldsymbol{w}:\nabla \boldsymbol{\phi} \, dx\text{, }\quad \boldsymbol{\phi} \in V_2.
		\end{split}
		\end{align}
		Furthermore, the operator $G:V_2\to V_2^*$ is G\^{a}teaux differentiable. 
	\end{theorem}
	\begin{proof}
		First, we consider the first summand on the right-hand side of equation $(\ref{Gright})$.
		
		The $o$-notation is the limit to zero in the following calculations. We prove pointwise convergence. For all $i,j \in \{1,...,N\}$, we calculate the Taylor expansion of $S^p$ in $D\boldsymbol{v}$ with the continuous derivative, see equation $(\ref{DerivativeSr})$. This yields almost everywhere
		\begin{align*}
		S^p_{ij}(D\boldsymbol{v}+tD\boldsymbol{w})
		&=S^p_{ij}(D\boldsymbol{v})+\sum_{k,\ell=1}^N(p-2)\big(|D\boldsymbol{v}|^2+\delta^2\big)^{(p-4)/2}(D\boldsymbol{v})_{ij}(D\boldsymbol{v})_{k\ell}t(D\boldsymbol{w})_{k\ell}
		\\
		&\quad  +\sum_{k,\ell=1}^N\big(|D\boldsymbol{v}|^2+\delta^2\big)^{(p-2)/2}I_{ik}I_{j\ell}t(D\boldsymbol{w})_{k\ell}
		+o(|tD\boldsymbol{w}|)
		\\
		&=S^p_{ij}(D\boldsymbol{v})+t(p-2)\big(|D\boldsymbol{v}|^2+\delta^2\big)^{(p-4)/2}(D\boldsymbol{v})_{ij}D\boldsymbol{v}:D\boldsymbol{w}
		\\
		&\quad+(|D\boldsymbol{v}|^2+\delta^2)^{(p-2)/2}t(D\boldsymbol{w})_{ij}+o(|tD\boldsymbol{w}|).
		\end{align*}
		This implies
		\begin{align*}
		&\quad \lim_{t \to 0}\frac{B}{t}\big(S^p_{ij}(D\boldsymbol{v}+tD\boldsymbol{w})-S^p_{ij}(D\boldsymbol{v})\big)
		\\
		&=(p-2)B(|D\boldsymbol{v}|^2+\delta^2)^{(p-4)/2}(D\boldsymbol{v})_{ij}D\boldsymbol{v}:D\boldsymbol{w}+B(|D\boldsymbol{v}|^2+\delta^2)^{(p-2)/2}(D\boldsymbol{w})_{ij}
		\end{align*}
		for all $i,j\in \{1,...,N\}$ almost everywhere. We obtained the pointwise convergence. To apply the dominated convergence theorem, we calculate a majorant. Using the Lipschitz continuity of $S^p$, see Lemma $\ref{LipschitzEstimate}$, and $C\in \mathbb{R}$, we conclude
		\begin{align}\label{Bound}
		&\quad \bigg|\frac{B}{t}\Big(S^p(D(\boldsymbol{v}+t\boldsymbol{w}))-S^p(D\boldsymbol{v})\Big):\nabla \boldsymbol{\phi}\bigg| \nonumber
		\\
		&\leq C \frac{|B|}{t}(\delta+|D(\boldsymbol{v}+t\boldsymbol{w})|+|D\boldsymbol{v}|)^{p-2}|D(\boldsymbol{v}+t\boldsymbol{w})-D\boldsymbol{v}|\,|\nabla \boldsymbol{\phi}| \nonumber
		\\
		&\leq C\delta^{p-2}\|B\|_{L^{\infty}(\Omega)}|D\boldsymbol{w}|\, |\nabla \boldsymbol{\phi}|.
		\end{align}
		The majorant is integrable as
		\begin{align*}
		\int_{\Omega} C\delta^{p-2}\|B\|_{L^{\infty}(\Omega)}|D\boldsymbol{w}|\, |\nabla \boldsymbol{\phi}|\, dx \leq C\delta^{p-2}\|B\|_{L^{\infty}(\Omega)}\|\boldsymbol{w}\|_{V_2}\|\boldsymbol{\phi}\|_{V_2}.
		\end{align*}
		The directional derivative for the second summand on the right-hand side of equation $(\ref{Gright})$ and the boundedness follow in the same ways if we use as the integration area $\Gamma_b$ instead of $\Omega$ and $V,\Phi$ with $V_{ij}:=I_{ij}v_j$ and $\Phi_{ij}:=I_{ij}\phi_j$ instead of $D\boldsymbol{v}$ and $D\boldsymbol{\phi}$. With this notation, we relate the vector-valued expressions to the matrix-valued expressions. The last summand in equation $(\ref{Derivative})$ follows trivially due to the linearity of $\boldsymbol{v}\mapsto \mu_0(\nabla \boldsymbol{v},\nabla \boldsymbol{\phi})$, $\boldsymbol{\phi}\in V_2$.
		
		Thus, $\boldsymbol{w}\mapsto G'(\boldsymbol{v};\boldsymbol{w})$ is a bounded linear operator. Hence, $G$ is G\^{a}teaux differentiable. 
	\end{proof}
	The operator $G$ has only a directional derivative on $V_2$ not on $V_p$:
	\begin{remark}\label{Differentiable}
		Let $\delta>0$. The second summand of $G'(\boldsymbol{v};\boldsymbol{w})$ in Theorem $\ref{continuity}$ is not well-defined for all $\boldsymbol{w},\boldsymbol{\phi}\in V_p$: Set $B\equiv 1$, $\boldsymbol{v}\equiv 0$. Then we have for all $\boldsymbol{w},\boldsymbol{\phi}\in V_2$
		\begin{align*} \int_{\Omega}B\big(|D\boldsymbol{v}|^2+\delta^2\big)^{(p-2)/2}D\boldsymbol{w}:\nabla \boldsymbol{\phi} \, dx
		=\int_{\Omega}\delta^{p-2}D\boldsymbol{w}:\nabla \boldsymbol{\phi}\, dx.
		\end{align*}
		Thus, the integral is not defined for suitable $\boldsymbol{w},\boldsymbol{\phi}\in V_p$.
	\end{remark}
	\section{Infinite-dimensional Newton's Method}\label{NewtonMethod} In this section, we state Newton's method in infinite dimensions and prove that we can calculate the Newton iterations. To our knowledge, this result is new as the G\^{a}teaux derivative exists only in all directions for $\mu_0>0$ and the combination of Newton's method and $\mu_0>0$ was not considered before. Newton's method is:
	
	Choose $\boldsymbol{v_0}\in V_2$ sufficiently close to the solution $\boldsymbol{\overline{v}}\in V_2$ of $G(\boldsymbol{\overline{v}})=0$.
	\newline
	For $k=0,1,2,...:$
	\newline
	Obtain $\boldsymbol{w_k}$ by solving
	\begin{align}\label{NewtonAlg}
	G'(\boldsymbol{v_k})\boldsymbol{w_k}=-G(\boldsymbol{v_k}),
	\end{align}
	and set $\boldsymbol{v_{k+1}}:=\boldsymbol{v_k}+\boldsymbol{w_k}$.
	
	The problem to calculate a Newton step, see equation $(\ref{NewtonAlg})$, is linear because we handle the linear problem $(\boldsymbol{w},\boldsymbol{\phi})\mapsto \langle G'(\boldsymbol{v})\boldsymbol{w},\boldsymbol{\phi}\rangle_{V_2^*,V_2}$ for all $\boldsymbol{w},\boldsymbol{\phi} \in V_2$. Hence, we can apply Lax-Milgram's Lemma. 
	\begin{lemma}\label{RegCondProved}
		Let $\mu_0>0$, $\delta>0$, and $\boldsymbol{v}\in V_2$. There exists exactly one solution $\boldsymbol{w}=\boldsymbol{w}(\boldsymbol{v})\in V_2$ such that equation $(\ref{NewtonAlg})$ is fulfilled. Moreover, we have
		\begin{align}\label{RegCond}
		\|G'(\boldsymbol{v})^{-1}\|_{\mathcal{L}(V_2,V_2^*)}\leq 1/\mu_0.
		\end{align}
	\end{lemma}
	\begin{proof}
		Let $\boldsymbol{v}\in V_2$. We show the continuity of $(\boldsymbol{w},\boldsymbol{\phi})\mapsto \langle G'(\boldsymbol{v})\boldsymbol{w},\boldsymbol{\phi}\rangle_{V_2^*,V_2}$ for all $\boldsymbol{w},\boldsymbol{\phi}\in V_2$. Since the arguments in the calculation of relation $(\ref{Bound})$ are also valid for the boundary term, we find $\tilde{C}\in \mathbb{R}$ with the trace operator, see \cite[Theorem 1.12]{Hinze2009}, such that
		\begin{align*}
		|\langle G'(\boldsymbol{v})\boldsymbol{w},\boldsymbol{\phi}\rangle_{V_2^*,V_2}|
		\leq \tilde{C}(\delta^{s-2}\|\tau\|_{L^{\infty}(\Gamma_b)}+\delta^{p-2}\|B\|_{L^{\infty}(\Omega)}+\mu_0) \|\boldsymbol{w}\|_{V_2} \|\boldsymbol{\phi}\|_{V_2}.
		\end{align*}
		We show the coercivity next. Let $\boldsymbol{w}\in V_2$. The last summand of the directional derivative of $G'$, see equation $(\ref{Derivative})$, is just
		\begin{align*}
		\mu_0(\nabla \boldsymbol{w},\nabla \boldsymbol{w})= \mu_0\|\boldsymbol{w}\|_{V_2}.
		\end{align*}
		We use $D\boldsymbol{w}:\nabla \boldsymbol{v}=D\boldsymbol{w}:D\boldsymbol{v}$ for arbitrary $\boldsymbol{v},\boldsymbol{w}\in V_2$, see equation $(\ref{Transpose})$, to conclude
		\begin{align*}
		&\quad\int_{\Omega}B\big(|D\boldsymbol{v}|^2+\delta^2\big)^{(p-2)/2}D\boldsymbol{w}:\nabla \boldsymbol{w}\, dx
		\\
		&\quad+\int_{\Omega}(p-2)B\big(|D\boldsymbol{v}|^2+\delta^2\big)^{(p-4)/2}(D\boldsymbol{v}:D\boldsymbol{w})\, (D\boldsymbol{v}:\nabla \boldsymbol{w)}\, dx
		\\
		&=
		\int_{\Omega}B\big(|D\boldsymbol{v}|^2+\delta^2\big)^{(p-2)/2}\bigg(|D\boldsymbol{w}|^2+(p-2)\frac{(D\boldsymbol{v}:D\boldsymbol{w})^2}{|D\boldsymbol{v}|^2+\delta^2}\bigg)\, dx
		\\
		&\geq
		\int_{\Omega}B\big(|D\boldsymbol{v}|^2+\delta^2\big)^{(p-2)/2}\bigg(|D\boldsymbol{w}|^2+(p-2)\frac{|D\boldsymbol{v}|^2\, |D\boldsymbol{w}|^2}{|D\boldsymbol{v}|^2}\bigg)\, dx
		\\
		&=
		\int_{\Omega}B\big(|D\boldsymbol{v}|^2+\delta^2\big)^{(p-2)/2}(p-1)|D\boldsymbol{w}|^2\, dx\geq 0.
		\end{align*}
		The same arguments for the boundary term yield
		\begin{align*}
		\int_{\Gamma_b}\tau\big(|\boldsymbol{v}|^2+\delta^2\big)^{(s-2)/2}|\boldsymbol{w}|^2+(s-2)\tau\big(|\boldsymbol{v}|^2+\delta^2\big)^{(s-4)/2}(\boldsymbol{v}\cdot \boldsymbol{w})\, (\boldsymbol{v}\cdot \boldsymbol{w}) \, ds \geq 0.
		\end{align*}
		Hence, the conditions for applying Lax-Milgram's Lemma are validated, and a unique solution for each Newton step, see equation $(\ref{NewtonAlg})$, exists. The coercivity constant is $\mu_0$. This implies the uniform bound
		\begin{align*} 
		\|(G'(\boldsymbol{v}))^{-1}\|_{\mathcal{L}(V_2,V_2^*)}\leq 1/\mu_0.
		\end{align*}	
	\end{proof}
	\section{Globalized Newton method}\label{GlobalSuperLinear}
	In this section, we prove global convergence of Newton's method with a step size control. To our knowledge, an analysis of Newton's method with a step size control was only performed in finite dimensions; for details of this analysis, see, e.g., \cite{Hirn2013}. Note that our minimization functional $J_{\mu_0,\delta}$ is not two times continuously differentiable as the second derivative is only a G\^{a}teaux derivative. From the applied perspective, the usage of approximations of exact step sizes, see Definition $\ref{ExactDefinition}$, could be interesting as they perform well in our numerical examples, see chapter $\ref{Experiments}$.
	
	To prove the convergence, we use the convex functional $J_{\mu_0,\delta}$ introduced in section $\ref{EquivalentProblem}$ for a step size control. We verify the equivalence of minimizing the convex functional and finding a weak solution of the $p$-Stokes equations:
	\begin{lemma}\label{AsMinimization}
		Let $\mu_0,\delta \in [0,\infty)$. Then, we have
		\begin{align*}
		J_{\mu_0,\delta}'(\boldsymbol{v})=0\Leftrightarrow J_{\mu_0,\delta}(\boldsymbol{v})=\min_{\boldsymbol{u}\in V_r}J_{\mu_0,\delta}(\boldsymbol{u}).
		\end{align*}
		with $r=2$ for $\mu_0>0$ and $r=p$ for $\mu_0=0$.
	\end{lemma}
	\begin{proof}
		The strict convexity for $(\mu_0,\delta)=(0,0)$ was proved in \cite{Chen2013} on $V_p$. Thus, the strict convexity follows for $\mu_0>0$ on the smaller space $V_2$. Thus, the case $\delta>0$ remains.	We verified in Lemma \ref{RegCondProved} positive definitness of $J''_{\mu_0,\delta}=G'$ for $(\mu_0,\delta)\in (0,\infty)^2$. This implies strict convexity. The case $\mu_0=0$ and $\delta>0$ is considered in \cite{Hirn2013}.
		
		Because the necessary first-order optimality condition is sufficient for strict convex functions, the claim follows.
	\end{proof}
	\subsection{Step size controls}
	In this subsection, we remind of step size controls. Step size controls guarantee convergence under certain conditions by setting $\boldsymbol{v_{k+1}}:=\boldsymbol{v_k}+\alpha_k \boldsymbol{w_k}$ with $\alpha_k\in (0,\infty)$ instead of $\boldsymbol{v_{k+1}}:=\boldsymbol{v_k}+\boldsymbol{w_k}$, see the begin of chapter $\ref{NewtonMethod}$. The choice of $\alpha_k$ should be such that the new functional value $J_{\mu_0,\delta}(\boldsymbol{v_k}+\alpha_k\boldsymbol{w_k})$ reduces compared to the old functional value $J_{\mu_0,\delta}(\boldsymbol{v_k})$ and a small reduction already implies convergence, see Lemma $\ref{GlobalConvergence}$ condition $c)$ for the mathematical details. A commonly used step size control that guarantees these properties under some conditions on $J_{\mu_0,\delta}$ is the Armijo step size, see for example \cite[section $2.2.1.1$]{Hinze2009}:
	\begin{definition}[Armijo step sizes]\label{ArmijoDefinition}
		Let $J_{\mu_0,\delta}:V_2\to \mathbb{R}$ be continuously differentiable, $\gamma \in (0,1)$, $\boldsymbol{v_k}\in V_2$ the point, and $\boldsymbol{w_k}\in V_2$ the direction. Determine the biggest $\alpha_k\in \{1,1/2,1/2^2,...\}$ such that
		\begin{align}\label{Armijo}
		J_{\mu_0,\delta}(\boldsymbol{v_k}+\alpha_k\boldsymbol{w_k})-J_{\mu_0,\delta}(\boldsymbol{v_k})\leq \alpha_k \gamma J'_{\mu_0,\delta}(\boldsymbol{v_k})\boldsymbol{w_k}.
		\end{align}
	\end{definition}
	For convex functions, we can try to find $\alpha_k$ that minimizes $J_{\mu_0,\delta}(\boldsymbol{v_k}+\alpha_k\boldsymbol{w_k})$:
	\begin{definition}[Exact step sizes]\label{ExactDefinition}
		Let $J_{\mu_0,\delta}:V_2\to \mathbb{R}$ be continuously differentiable, $\boldsymbol{v_k}\in V_2$ the point, and $\boldsymbol{w_k}\in V_2$ the direction. Determine $\alpha_k$ with
		\begin{align*}
		J_{\mu_0,\delta}(\boldsymbol{v_k}+\alpha_k \boldsymbol{w_k})=\min_{\alpha\in  [0,\infty)}J_{\mu_0,\delta}(\boldsymbol{v_k}+\alpha \boldsymbol{w_k}).
		\end{align*}
	\end{definition}
	In our case, we can only approximately solve this problem, which we discuss in the numerical experiment.
	\subsection{Global convergence}In this subsection, we verify global convergence. 
	We obtain global convergence by employing a step size control. The step sizes are calculated by reducing the functional $J_{\mu_0,\delta}$, see Definition $\ref{ConvexFunction}$, in each iteration. We verify the conditions for the following convergence result:
	\begin{lemma}\label{GlobalConvergence}
		Let $J_{\mu_0,\delta}:V_2\to \mathbb{R}$ be continuously Fr\'{e}chet differentiable, bounded from below, $\boldsymbol{v_0}\in V_2$, $\boldsymbol{v_{k+1}}:=\boldsymbol{v_k}+\alpha_k\boldsymbol{w_k}$ with the direction $\boldsymbol{w_k}\in V_2$, and the step sizes $\alpha_k \in (0,\infty)$. We additionally assume that we have
		\begin{itemize}
			\item[a)] descent directions: $J'_{\mu_0,\delta}(\boldsymbol{v_k})\boldsymbol{w_k}<0$,
			\vspace{\baselineskip}
			\item[b)] the angle condition: $\eta \|J'_{\mu_0,\delta}(\boldsymbol{v_k})\|_{V_2^*}\|\boldsymbol{w_k}\|_{V_2}\leq - J_{\mu_0,\delta}'(\boldsymbol{v_k})\boldsymbol{w_k}\quad$ independently of $k$ for fixed $\eta \in(0,1)$,
			\vspace{\baselineskip}
			\item[c)] large enough step sizes:
			$
			\left\{\begin{array}{c}
			J_{\mu_0,\delta}(\boldsymbol{v_k}+\alpha_k\boldsymbol{w_k})<J_{\mu_0,\delta}(\boldsymbol{v_k})\quad \text{for all }k
			\\
			\text{and }J_{\mu_0,\delta}(\boldsymbol{v_k}+\alpha_k\boldsymbol{w_k})-J_{\mu_0,\delta}(\boldsymbol{v_k})\to 0\quad \text{for }k \to \infty
			\end{array}\right\}
			\\
			\phantom{space}
			\\
			\text{imply } \frac{J_{\mu_0,\delta}'(\boldsymbol{v_k})\boldsymbol{w_k}}{\|\boldsymbol{w_k}\|_{V_2}}\to 0 \text{ for }k\to \infty.$
		\end{itemize}
		Then, we have
		\begin{align*}
		J_{\mu_0,\delta}'(\boldsymbol{v_k})\to 0\quad \text{for }k\to \infty.
		\end{align*}
	\end{lemma}
	\begin{proof}
		See \cite{Hinze2009}.
	\end{proof}
	The first two statements are easy to verify: 
	\begin{lemma}
		Let $\mu_0>0$, $\delta>0$. Newton steps fulfill the angle condition for the functional $J_{\mu_0,\delta}$, see Definition $\ref{ConvexFunction}$, and are descent directions.
	\end{lemma}
	\begin{proof}
		We calculated in Lemma $\ref{RegCondProved}$ the coercivity constant $\mu_0$ and the continuity constant $C$ independent of $k$ such that
		\begin{align}\label{BoundJprimeprime}
		\mu_0\leq \|J_{\mu_0,\delta}''(\boldsymbol{v_k})\|_{\mathcal{L}(V_2,V_2^*)}\leq C.
		\end{align}
		In \cite[Theorem 1.2]{Polyak2007}, the proof of the angle condition is done with $\eta\geq \mu_0/C$ for the finite-dimensional case. The proof for the infinite-dimensional case is identical. The angle condition implies that the Newton steps are descent directions.
	\end{proof}
	To verify large enough step sizes, we need the following lemma:
	\begin{lemma}\label{Monotonicity}
		Let $J_{\mu_0,\delta}'$ be uniformly continuous. Let the step sizes $(\alpha_k)_k$ be Armijo step sizes and let the direction $\boldsymbol{w_k}$ fulfill
		\begin{align*}
		\|\boldsymbol{w_k}\|_{V_2}\geq \varphi\bigg(\frac{-J_{\mu_0,\delta}'(\boldsymbol{v_k})\boldsymbol{w_k}}{\|\boldsymbol{w_k}\|_{V_2}}\bigg)
		\end{align*}
		with some $\varphi:[0,\infty)\to [0,\infty)$ monotonically increasing and satisfying $\varphi(t)>0$ for all $t>0$. Then the step sizes $(\alpha_k)_k$ are admissible.
	\end{lemma}
	\begin{proof}
		See \cite{Hinze2009} Lemma $2.3$.
	\end{proof}
	\begin{lemma}\label{Admissible}
		Let $\mu_0>0$, $\delta>0$, and $|\Gamma_d|>0$. For Newton's method, the Armijo step sizes, see Definition $\ref{ArmijoDefinition}$, are admissible.
	\end{lemma}
	\begin{proof}
		We want to apply Lemma $\ref{Monotonicity}$. We verify the conditions. The function $J'_{\mu_0,\delta}$ is Lipschitz continuous, because the operator $A$ is Lipschitz continuous, see Lemma $\ref{Lipschitz}$, and $(\boldsymbol{v},\boldsymbol{w})\mapsto \mu_0(\boldsymbol{v},\boldsymbol{w})$ is trivially Lipschitz continuous. Therefore, $J_{\mu_0,\delta}'$ is uniformly continuous. We already proved that the directions $\boldsymbol{w_k}$ are descent directions. It remains to show that the descent directions are not too short. We know with the Newton steps and the right inequality in relation $(\ref{BoundJprimeprime})$
		\begin{align*}
		&\quad C\|\boldsymbol{w_k}\|_{V_2}^2\geq \langle J_{\mu_0,\delta}''(\boldsymbol{v_k})\boldsymbol{w_k},\boldsymbol{w_k}\rangle_{V_2^*,V_2}=-\langle J_{\mu_0,\delta}'(\boldsymbol{v_k}),\boldsymbol{w_k}\rangle_{V_2^*,V_2}
		\\
		&\Leftrightarrow \|\boldsymbol{w_k}\|_{V_2}\geq \frac{1}{C}\frac{-\langle J_{\mu_0,\delta}'(\boldsymbol{v_k}),\boldsymbol{w_k}\rangle_{V_2^*,V_2}}{\|\boldsymbol{w_k}\|_{V_2}}.
		\end{align*}
		Therefore, the step size is bounded from below by the monotonically increasing function $\varphi:[0,\infty)\to [0,\infty)$, $\varphi(t)=t/C$ with $\varphi(t)>0$ for $t>0$. Hence, the step sizes are large enough, and we can apply Lemma $\ref{Monotonicity}$.
	\end{proof}
	\begin{corollary}
		Let $\mu_0>0$, $\delta>0$, and $|\Gamma_d|>0$. Newton's method with Armijo step sizes is globally convergent.
	\end{corollary}
	\begin{proof}
		We already explained why the step sizes are admissible, see Lemma $\ref{Admissible}$. Moreover, the directions $\boldsymbol{w_k}$ are descent directions. Furthermore, we validated the angle condition. Trivially, the functional $J_{\mu_0,\delta}$ is bounded from below. Thus, all conditions for Lemma $\ref{GlobalConvergence}$ are fulfilled. Hence, we conclude $G(\boldsymbol{v_k})=J_{\mu_0,\delta}'(\boldsymbol{v_k})\to 0\in V_2^*$.
		
		Let $\boldsymbol{v^*}\in V_2$ be the solution of $G(\boldsymbol{v^*})=0$. For $\boldsymbol{v_k}=\boldsymbol{v^*}$ the claim is clear. For $\boldsymbol{v_k}\neq \boldsymbol{v^*}$, the strict monotonicity of $A$, see Lemma $\ref{StrictMonotone}$, implies
		\begin{align*}
		\mu_0 \|\boldsymbol{v_k}-\boldsymbol{v^*}\|_{V_2}^2&\leq
		\langle A(\boldsymbol{v_k})-A(\boldsymbol{v}^*),\boldsymbol{v_k}-\boldsymbol{v^*}\rangle_{V_2^*,V_2}
		\\
		&= \langle G(\boldsymbol{v_k})-G(\boldsymbol{v^*}),\boldsymbol{v_k}-\boldsymbol{v^*}\rangle_{V_2^*,V_2}
		\\
		&\leq \|G(\boldsymbol{v_k})-G(\boldsymbol{v^*})\|_{\mathcal{L}(V_2,V_2^*)}\|\boldsymbol{v_k}-\boldsymbol{v^*}\|_{V_2}
		\\
		\Leftrightarrow \mu_0\|\boldsymbol{v_k}-\boldsymbol{v^*}\|_{V_2}&\leq \|G(\boldsymbol{v_k})-G(\boldsymbol{v^*})\|_{\mathcal{L}(V_2,V_2^*)}.
		\end{align*}
		Due to $G(\boldsymbol{v_k})=J'_{\mu_0,\delta}(\boldsymbol{v_k})\to 0$ for $k\to \infty$ and $G(\boldsymbol{v^*})=0$ follows $\boldsymbol{v_k}\to \boldsymbol{v^*}\in V_2$.
	\end{proof}
	\subsection{Convergence to the non regularized problem}\label{ConvergenceRegularisation}
	In this subsection, we argue that the solution $\boldsymbol{v}=\boldsymbol{v_{\mu_0,\delta}}\in V_2$ converges to $\boldsymbol{v_{0,0}}\in V_2$. We only have to assume $\boldsymbol{v_{0,0}}\in V_2$. From \cite{Chen2013}, we only know $\boldsymbol{v_{0,0}}\in V_p$.
	\begin{theorem}[Convergence for smooth solutions]
		Let $|\Gamma_d|>0$ and $\boldsymbol{v_{\mu_0,\delta}}$ be the solution of $G_{\mu_0,\delta}(\boldsymbol{v})=0$ with the variables $\mu_0,\delta \in (0,\infty)$ or $(\mu_0,\delta)=(0,0)$ as in the definition of the operator $G$. Assume $\boldsymbol{v_{0,0}}\in V_2$. Then there exists $\tilde{c}\in \mathbb{R}$ with
		\begin{align*}
		\|\boldsymbol{v_{0,0}}-\boldsymbol{v_{\mu_0,\delta}}\|_{V_2}^2\leq \tilde{c}(|\Omega|\delta^p+|\Gamma_b|\delta^s+\mu_0).
		\end{align*}
	\end{theorem}
	\begin{proof}
		We follow the idea in \cite[Theorem 4.1]{Hirn2013}. We set $\boldsymbol{\tilde{v}}:=\boldsymbol{v_{0,0}}-\boldsymbol{v_{\mu_0,\delta}}$ to shorten the notation. We use in the upcoming calculation the monotonicity of $S^p$ and $S^s$, $J_{\mu_0,\delta}'(\boldsymbol{v_{\mu_0,\delta}})=0$, and the main theorem of calculus:
		\begin{align*}
		&\quad \frac{1}{2}\|\boldsymbol{\tilde{v}}\|_{V_2}^2
		\\
		&= 
		\int_0^1\int_{\Omega}\frac{1}{r}|r\nabla \boldsymbol{\tilde{v}}|^2\, dx\, dr
		\\
		&\leq 
		\int_0^1\int_{\Omega}\frac{1}{r}\big(S^p(D\boldsymbol{v_{\mu_0,\delta}}+rD\boldsymbol{\tilde{v}})-S^p(D\boldsymbol{v_{\mu_0,\delta}})\big)
		:(D\boldsymbol{v_{\mu_0,\delta}}+rD\boldsymbol{\tilde{v}}-D\boldsymbol{v_{\mu_0,\delta}})\, dx\, dr
		\\
		&\quad +\int_0^1\int_{\Gamma_b}\frac{1}{r}\big(S^s(\boldsymbol{v_{\mu_0,\delta}}+r\boldsymbol{\tilde{v}})-S^s(\boldsymbol{v_{\mu_0,\delta}})\big)\cdot(\boldsymbol{v_{\mu_0,\delta}}+r\boldsymbol{\tilde{v}}-\boldsymbol{v_{\mu_0,\delta}})\, ds\, dr
		\\
		&\quad +\int_0^1\int_{\Omega}\frac{1}{r}|r\nabla \boldsymbol{\tilde{v}}|^2\, dx\, dr
		\\
		&=\int_0^1\frac{1}{r}\big(J'_{\mu_0,\delta}(\boldsymbol{v_{\mu_0,\delta}}+r\boldsymbol{\tilde{v}})-J_{\mu_0,\delta}'(\boldsymbol{v_{\mu_0,\delta}})\big)(\boldsymbol{v_{\mu_0,\delta}}+r\boldsymbol{\tilde{v}}-\boldsymbol{v_{\mu_0,\delta}})\, dr
		\\
		&=\int_0^1J'_{\mu_0,\delta}(\boldsymbol{v_{\mu_0,\delta}}+r\boldsymbol{\tilde{v}})
		\boldsymbol{\tilde{v}}\, dr=J_{\mu_0,\delta}(\boldsymbol{v_{0,0}})-J_{\mu_0,\delta}(\boldsymbol{v_{\mu_0,\delta}}).
		\end{align*}
		In the last step, we inserted $\boldsymbol{\tilde{v}}$. We use
		\begin{align*}
		(|D\boldsymbol{v}|^2+\delta^2)^{r/2}=\|(D\boldsymbol{v},\delta)\|_2^r\leq \|(D\boldsymbol{v},\delta)\|_r^r=|D\boldsymbol{v}|^r+\delta^r\quad \text{ for }r\in \{s,p\}
		\end{align*}
		to bound the functional $J_{\mu_0,\delta}$ by $J_{\mu_0,0}$ for all $\boldsymbol{v}\in V_2$ with   \begin{align*}
		J_{\mu_0,\delta}(\boldsymbol{v})&=\int_{\Omega}\frac{1}{p}(|D\boldsymbol{v}|^2+\delta^2)^{p/2}\, dx+\int_{\Gamma_b}\frac{1}{s}(|\boldsymbol{v}|^2+\delta^2)^{s/2}\, ds+\mu_0 \|\boldsymbol{v}\|_{V_2}^2
		\\
		&\leq \frac{1}{p}\int_{\Omega}|D\boldsymbol{v}|^p+\delta^p\, dx
		+\frac{1}{s}\int_{\Gamma_b}|\boldsymbol{v}|^s+\delta^s\, ds+\mu_0 \|\boldsymbol{v}\|_{V_2}^2
		\\
		&= J_{0,0}(\boldsymbol{v})+|\Omega|\delta^p+|\Gamma_b|\delta^s+\mu_0\|\boldsymbol{v}\|_{V_2}^2.
		\end{align*}
		Hence, we obtain with the minimizer $\boldsymbol{v_{0,0}}$ of $J_{0,0}$ and the definitions of $J_{0,0}$ and $J_{\mu_0,\delta}$
		\begin{align*}
		&\quad J_{\mu_0,\delta}(\boldsymbol{v_{0,0}})-J_{\mu_0,\delta}(\boldsymbol{v_{\mu_0,\delta}})
		\\
		&\leq 
		J_{0,0}(\boldsymbol{v_{0,0}})+|\Omega|\delta^p+|\Gamma_b|\delta^s+\mu_0\|\boldsymbol{v_{0,0}}\|_{V_2}^2-J_{\mu_0,\delta}(\boldsymbol{v_{\mu_0,\delta}})
		\\
		&\leq J_{0,0}(\boldsymbol{v_{\mu_0,\delta}})+|\Omega|\delta^p+|\Gamma_b|\delta^s+\mu_0 \|\boldsymbol{v_{0,0}}\|_{V_2}^2-J_{\mu_0,\delta}(\boldsymbol{v_{\mu_0,\delta}})
		\\
		&=\frac{1}{p}\int_{\Omega}|D\boldsymbol{v_{\mu_0,\delta}}|^p-(\delta^2+|D\boldsymbol{v_{\mu_0,\delta}}|^2)^{p/2}\, dx+
		\frac{1}{s}\int_{\Gamma_b}|\boldsymbol{v_{\mu_0,\delta}}|^s-(\delta^2+|\boldsymbol{v_{\mu_0,\delta}}|^2)^{s/2}\, ds
		\\
		&\quad+|\Omega|\delta^p+|\Gamma_b|\delta^s+\mu_0\big(\|\boldsymbol{v_{0,0}}\|_{V_2}^2-\|\boldsymbol{v_{\mu_0,\delta}}\|_{V_2}^2\big)
		\\
		&\leq |\Omega|\delta^p+|\Gamma_b|\delta^s+\mu_0\|\boldsymbol{v_{0,0}}\|_{V_2}^2.
		\end{align*}
		Thereby, we obtain with $\tilde{c}\in \mathbb{R}$
		\begin{align*}
		\|\boldsymbol{v_{0,0}}-\boldsymbol{v_{\mu_0,\delta}}\|_{V_2}^2\leq \tilde{c}(|\Omega|\delta^p+|\Gamma_b|\delta^s+\mu_0).
		\end{align*}
	\end{proof}
	\section{Numerical experiment}\label{Experiments}
	In this section, we describe two numerical experiments: ISMIP-HOM $B$ formulated in \cite{Pattyn2008} and a sliding block. We test these experiments with Newton's method with Armijo step sizes and compare them with approximately exact step sizes that we describe in the next subsection. We also use the approximately exact step sizes to modify the Picard iteration. We implemented the experiments in \textit{FEniCS}, \cite{Logg2012}, version $2019.1.0$. \textit{FEniCS} is a tool that solves partial differential equations with finite elements. We use the Taylor-Hood element $P_2-P_1$ for the velocity $\boldsymbol{v}$ and the pressure $\pi$. In all experiments, we calculate the initial guess by replacing $\big(|D\boldsymbol{v}|^2+\delta^2\big)^{(p-2)/2}$ with $10^6$ and solving the resulting Stokes problem. 
	
	Let $\boldsymbol{v}\in V_2$. Then $G(\boldsymbol{v})\in V_2^*$ with the norm
	\begin{align*}
	\|G(\boldsymbol{v})\|_{V_2^*}=\sup_{\boldsymbol{\phi}\in V_2, \|\boldsymbol{\phi}\|_{V_2}=1}\langle G(\boldsymbol{v}),\boldsymbol{\phi}\rangle_{V_2^*,V_2}.
	\end{align*}
	To evaluate this norm, we use the Riesz isomorphism:
	\begin{definition}[Norm of the Riesz isomorphism]\label{ResidualNormDef}
		Let $|\Gamma_d|>0$, $\mu_0>0$, $\delta>0$, and $\boldsymbol{v}\in V_2$. Let $\boldsymbol{\tilde{v}}\in V_2$ be the solution of
		\begin{align*}
		\int_{\Omega}\nabla \boldsymbol{\tilde{v}}:\nabla \boldsymbol{\phi}\, dx = \langle G(\boldsymbol{v}),\boldsymbol{\phi}\rangle_{V_2^*,V_2}\quad \text{for all }\boldsymbol{\phi}\in V_2.
		\end{align*}
		Then the norm of the Riesz isomorphism is $ries\in [0,\infty)$ is
		\begin{align*}
		ries^2:=\int_{\Omega}|\nabla \boldsymbol{\tilde{v}}|^2\, dx +\int_{\Gamma_b}|\boldsymbol{\tilde{v}}|^2\, ds
		\end{align*}
		with $\boldsymbol{\tilde{v}}\cong G(\boldsymbol{v})\in V_2^*$.
	\end{definition}
	If the norm of the Riesz isomorphism is small for $\boldsymbol{v}\in V_2$, it follows $G(\boldsymbol{v})\approx 0$.
	\subsection{Numerical solvers and computational effort}\label{SolversAndEffort}
	The Picard iteration was suggested for modeling glaciers in \cite{Colinge1999} and is used in ice models like \textit{ISSM}, see \cite{Larour2012}. The Picard iteration is: For given $\boldsymbol{v_k}\in V_2$ find $\boldsymbol{v_{k+1}}\in V_2$ with
	\begin{align*}
	(B(|D\boldsymbol{v_k}|^2+\delta^2)^{(p-2)/2}D\boldsymbol{v_{k+1}},\nabla \boldsymbol{\phi})
	&+(\tau(|\boldsymbol{v_k}|^2+\delta^2)^{(s-2)/2}\boldsymbol{v_{k+1}},\boldsymbol{\phi})
	\\
	&+\mu_0(\nabla \boldsymbol{v_{k+1}},\nabla \boldsymbol{\phi})=(\rho \boldsymbol{g},\boldsymbol{\phi})
	\end{align*}
	for all $\boldsymbol{\phi} \in V_2$. This algorithm is said to be often globally convergent in practice but slow, \cite{Fraters2019}. It is, for example, used for the Navier-Stokes equations with proved convergence theory, \cite{Girault1986}. We will compare the following algorithms:
	\begin{itemize}
		\item The Picard iteration.
		\item The Picard iteration with approximations of exact step sizes, see Definition $\ref{ExactDefinition}$, with $\boldsymbol{w_k}:=\boldsymbol{v_{k+1}}-\boldsymbol{v_k}$.
		\item Newton's method with approximations of exact step sizes, see Definition $\ref{ExactDefinition}$.
		\item Newton's method with Armijo step sizes, see Definition $\ref{ArmijoDefinition}$.
	\end{itemize}
	We can calculate the exact step size arbitrarily precise because we have a convex functional. Let $0\leq a<b<\infty$. For the actual velocity $\boldsymbol{v_k}$ and the direction $\boldsymbol{w_k}$, we can calculate
	\begin{align*}
	\min_{t \in [a,b]}J_{help}(t):=\min_{t \in [a,b]}J_{\mu_0,\delta}(\boldsymbol{v_k}+t\boldsymbol{w_k}).
	\end{align*}
	We use a simple bisection and set $b:=(a+b)/2$ for $J_{help}'((a+b)/2)\geq 0$ and $a:=(a+b)/2$ for $J_{help}'((a+b)/2)<0$ until we obtain the wanted accuracy. Finally, we set $t:=(a+b)/2$. The chain rule yields for $t\in (a,b)$
	\begin{align}\label{DerivativeHelpFunc}
	J_{help}'(t)=\langle G(\boldsymbol{v_k}+t\boldsymbol{w_k}),\boldsymbol{w_k}\rangle_{V_2^*,V_2}.
	\end{align}
	Furthermore, we can do these calculations with mixed elements:
	\begin{remark}
		The step-size control is identical with mixed elements.	Only the calculation of the direction changes.
	\end{remark}
	\begin{proof}
		We calculate the initial guess $(\boldsymbol{v_0},\pi_0)\in(W_2,L^2_0(\Omega))$ by replacing $\big(|D\boldsymbol{v}|^2+\delta^2\big)^{(p-2)/2}$ with $10^6$ and solving the resulting Stokes problem. This initial guess is divergence-free. Then, we calculate a divergence-free direction $\boldsymbol{w_0}$ for the velocity and a direction $\psi_0$ for the pressure. Hence, we have for the convex functional
		\begin{align*}
		J_{\mu_0,\delta}'(\boldsymbol{v_0})\boldsymbol{w_0}=\langle  G(\boldsymbol{v_0}),\boldsymbol{w_0}\rangle_{V_2^*,V_2}
		=\langle  G(\boldsymbol{v_0}),\boldsymbol{w_0}\rangle_{V_2^*,V_2}-(\pi_0,\mathrm{div}\boldsymbol{w_0})-(\mathrm{div}\boldsymbol{v_0},\psi_0).
		\end{align*}
		Additionally, our iteratives $\boldsymbol{v_k}$ are divergence-free as a linear combination of the divergence-free $\boldsymbol{v_{k-1}}$ and $\boldsymbol{w_{k-1}}$.
		
		Thus, the step size control is identical with mixed elements.
	\end{proof}
	Now, we consider the computational effort. The computation consists of three parts: Assembling the matrices, solving the linear system of equations, and calculating the step size.
	
	In \cite{Isaac2015}, multigrid methods are used to solve the $p$-Stokes equations, even on real-world examples like the Antarctic ice sheet. In \cite{Brown2013}, the computational effort of multigrid method is discussed. They measured that evaluating the residual takes less computation time than, e.g., assembling the Jacobi matrix or necessary matrix-vector multiplications. For their results, they used a simplified version of the $p$-Stokes equations. As calculating the residual is not the computationally expensive part also calculating the step size control is computationally affordable compared to the other parts.
	\subsection{Computational domain}
	A unit square is too simple to represent real-world glaciers. Instead, in the ISMIP-HOM experiments, see \cite{Pattyn2008}, a sinusoidal bedrock, see Fig. \ref{domain}, is suggested, and Dirichlet boundary conditions are imposed at the bedrock. We set $\alpha:=0.5$ \textup{$^{\circ}$}, $L:=5000$ \textup{m}, and define the upper and lower boundary of the domain by $z_s:[0,L]\to \mathbb{R}$, $z_b:[0,L]\to \mathbb{R}$,
	\begin{align*}
	z_s(x)=-\tan(\alpha)x\text{, }\quad \text{and } z_b(x)=z_s(x)-1000+500\sin(\omega x)
	\end{align*}
	with $\omega=2\pi/L$. The experiment needs periodic boundary conditions on the left and the right side of the domain and $\sigma \cdot \boldsymbol{n}=0$ at the surface. 
	\subsection{Parameters}
	In this subsection, we discuss setting constants and forcing boundary conditions. We set the constants for the experiment corresponding to the ISMIP-HOM experiments, see \cite{Pattyn2008}: The ice parameter $B:=0.5\cdot (10^{-16})^{-1/3}$ \textup{Pa$^{-3}$a$^{-1}$}, the density $\rho:=910$ \textup{kg m$^{-3}$}, the gravitational acceleration $\mathbf{g}:=(0,-9.81)$ \textup{m s$^{-2}$}, and the seconds per year $31\, 556\, 926$. The nonlinear viscosity is $(0.5|D\boldsymbol{v}|^2+\delta^2)^{(p-2)/2}$ with $p:=4/3$. The factor $0.5$ within the nonlinear term is given by \cite{Pattyn2008}. However, we could modify $B$ and $\delta$ to formulate the equations as in our convergence analysis before. 
	
	Periodic boundary conditions have some difficulties in implementation for unstructured grids. Moreover, they are not necessary for real-world applications. Thus, we extend our domain by three copies to the left and the right. We apply Dirichlet boundary conditions at these boundaries. This approach is suggested in the Supplement of \cite{Pattyn2008}. 
	\begin{figure}[h]
		\centering
		\includegraphics[width=\textwidth]{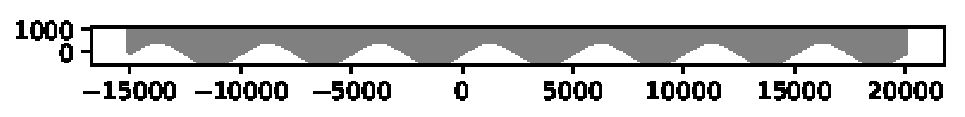}
		\caption{Considered domain in our experiment.}
		\label{domain}
	\end{figure}
	The resulting domain is shown in Fig. $\ref{domain}$. 
	
	We rotate the gravity to obtain a horizontal upper surface. Formally, we should rotate the left and right boundaries, too. However, this only changes the position in $x$-direction by a maximum of $1500 \textup{ m}\cdot \sin(0.5$\textup{$^{\circ})$}$\approx 13\textup{ m}$. 
	This is neglectable compared to the horizontal extent of $35\, 000 \textup{ m}$. Thus, we simplify the domain with vertical left and right boundaries and a horizontal surface boundary. Moreover, this approach is more flexible because we can simply change the angle $\alpha$ without changing the domain to generate different experiments. Furthermore, changes between using periodic boundary conditions and the variant with copies of the glacier are easier. 
	
	The value $\delta^2$ should be small compared to typical values of $0.5|D\boldsymbol{v}|^2$. The range of $|D\boldsymbol{v}|$ is typically between $10^{-7}$ \textup{/s} and $10^{-11}$ \textup{/s}. The calculations are done in years. Hence, we obtain as a necessary condition
	\begin{align*}
	\delta^2 \leq 0.5(3.16\cdot 10^{-4}\textup{/a})^2\cdot \mathrm{eps}
	\end{align*}
	with the machine precision $\mathrm{eps}:= 10^{-16}$. By rounding down to a factor of $10$, we obtain $\delta:=10^{-12}\textup{/a}$. We should choose $\mu_0$ such that 
	\begin{align*}
	\mu_0 \leq(0.5|D\boldsymbol{v}|^2+\delta^2)^{(p-2)/2}\cdot \mathrm{eps}
	\end{align*}
	is fulfilled. We obtain an upper bound by inserting the maximum typical value of $|D\boldsymbol{v}|$ for $|D\boldsymbol{v}|$ and $\delta$. We conclude
	\begin{align*}
	\mu_0 \leq (2\cdot (10^{-7}\cdot 31\, 556\, 926 \textup{/a})^2)^{(4/3-2)/2}\cdot 10^{-16}\approx 3.69\cdot 10^{-17}\textup{a}^{2/3}.
	\end{align*}
	Consequently, we choose $\mu_0:=10^{-17}$ $\textup{kg a/}(\textup{m}\textup{s}^2)$. The unit for $\mu_0$ follows from $[\mu_0 |\nabla \boldsymbol{v}|^2]=[-\rho \boldsymbol{g}\cdot \boldsymbol{v}]$.
	\subsection{Results and interpretation}
	The ISMIP-HOM $B$ experiment measures the norms of the surface velocity $v_r$:
	\begin{align*}
	v_r:=\sqrt{(v_1\cdot \cos(\alpha))^2-(v_2\cdot \sin(\alpha))^2}.
	\end{align*}
	On the Fig. $\ref{VelField}$ are the surface velocities for all used algorithms restricted to the original glacier $x\in [0,5000]$. The surface velocities are approximately the same for all our methods. Moreover, the velocities are similar to \cite[Figure 6]{Pattyn2008}.
	\begin{figure}[h]
		\includegraphics[width=\textwidth]{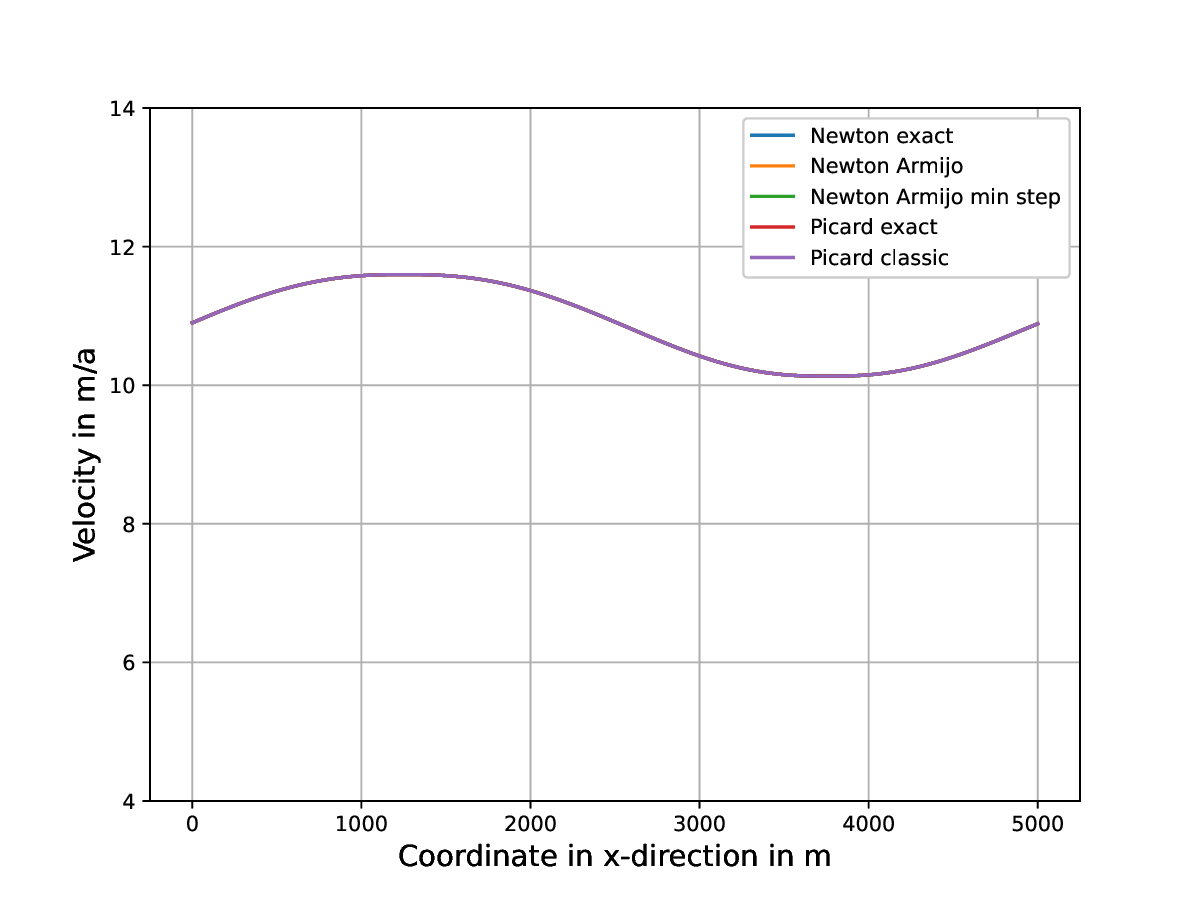}
		\caption{The velocity norms at the surface are shown in the left figure. All methods produce nearly the same surface velocity as the result.} 
		\label{VelField}
	\end{figure}
In Fig. $\ref{Residuals}$, we see the relative norm of the Riesz isomorphism, which compares the actual norm of the Riesz isomorphism of the iteration with the first calculated norm of the Riesz isomorphism, see Definition $\ref{ResidualNormDef}$. Newton's method with Armijo step sizes reduces the norm of the Riesz isomorphism compared to the initial norm of the Riesz isomorphism by approximately $5\cdot 10^3$. Newton's method with Armijo step sizes starts with a quadratic convergence. After a few iterations, this tends to a linear convergence. Then, the norm of the Riesz isomorphism is constant. The norm of the Riesz isomorphism of the Picard iteration decreases linearly but at a slower rate than Newton's method with Armijo step sizes. Newton's method and the Picard iteration with approximately exact step sizes converge similarly fast as Newton's method with Armijo step sizes.

In \cite{Hirn2013}, it was discussed that a higher value of $\delta$ leads to higher accuracy and a lower number of necessary iterations. We reproduce this result in Fig. $\ref{ResidualsDifferentDelta}$. Additionally, we see that the velocity field at the surface is nearly identical for $\delta:=10^{-4}$. With $\delta:=10^{-4}$, Newton's method is much better than the Picard iteration for both step size controls, see Fig. $\ref{DifferentDelta}$.
\begin{figure}[h]
	\includegraphics[scale=0.258]{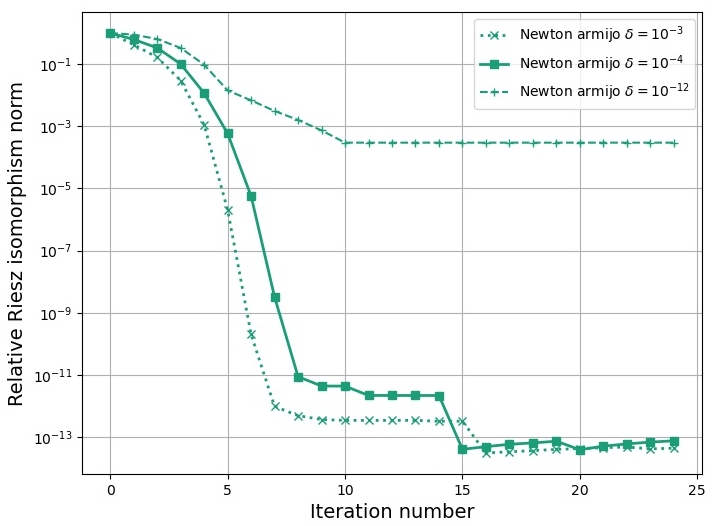}
	\includegraphics[scale=0.258]{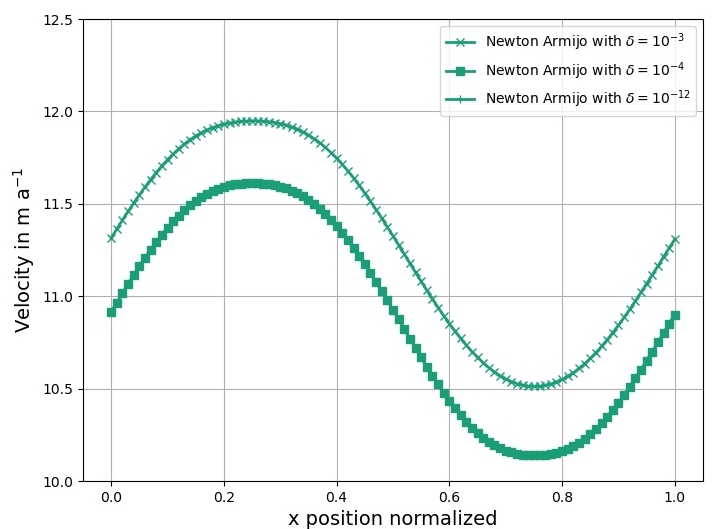}
	\caption{Left: The relative norms of the Riesz isomorphism are visualized for different values of $\delta$. Right: The velocity norms at the surface are shown in the right figure.}
	\label{ResidualsDifferentDelta}
\end{figure}
\begin{figure}[h]
	\includegraphics[width=\textwidth]{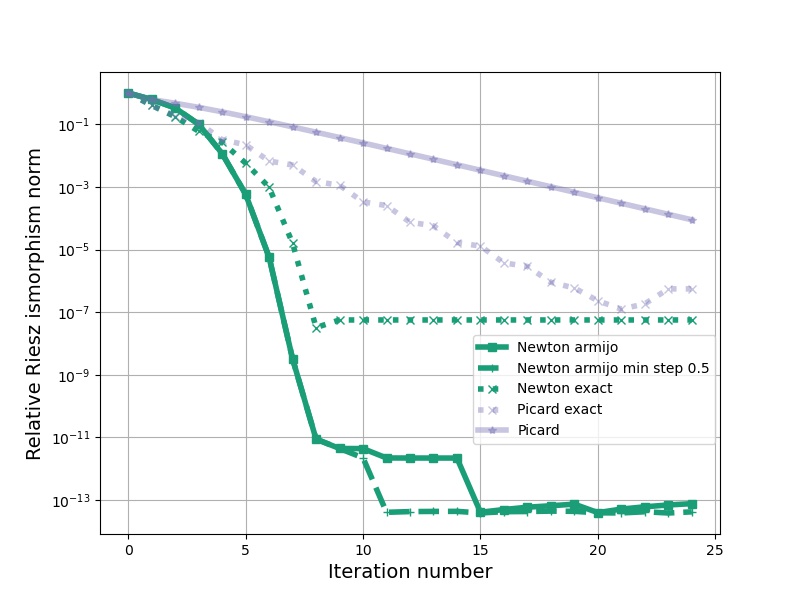}
	\caption{The relative norms of the Riesz isomorphism are visualized for $\delta:=10^{-4}$.}
	\label{DifferentDelta}
\end{figure}
To verify our claim regarding the computation time, see subsection $\ref{SolversAndEffort}$, we calculated the mean and the standard deviation for $100$ iterations for each method, see Table $\ref{ComputationTimes}$. In our experiment, the computation time with a step size control is within the standard deviation of the computation time for the Picard iteration. Moreover, the step size control needs only $1$ \% of the whole computation time for each iteration. Another important result is that the Picard iteration and Newton's method need comparable computation times.

\begin{table}
	\begin{tabular}{|c|c|c|c|c|}
		\hline
		& \multicolumn{2}{c|}{Computation time each iteration} & \multicolumn{2}{c|}{Computation time step size} \\
		\hline
		& Mean & Standard deviation & Mean & Standard deviation \\
		\hline
		Picard & 98.39 & 2.39 & - & - \\
		\hline
		Picard with exact step sizes & 99.47 & 2.16 & 1.00 & 0.01 \\
		\hline
		Newton with Armijo step sizes & 98.89 & 2.40 & 0.41 & 0.12 \\
		\hline
		Newton with Armijo step sizes & 98.55 & 2.29 & 0.08 & 0.01 \\
		and minimum step size=0.5 & & & & \\
		\hline
		Newton with exact step sizes & 99.71 & 2.23 & 1.01 & 0.01 \\
		\hline
	\end{tabular}
	\caption{Computation time in seconds.}
	\label{ComputationTimes}
\end{table}

We fixed the number of integral evaluations for the Armijo step size to a maximum of $20$ integral evaluations for each iteration and for the exact step sizes to $25$. Newton's method with a minimum step size of $0.5$ needs one or two iterations. Thus, Table $\ref{ComputationTimes}$ implies that both step size controls need a similar computation time for each iteration. We did not try to reduce the number of iterations for the step size because they are not important compared to solving the linear system of equations.
\subsection{Numerical experiment with friction}
In this subsection, we describe an experiment with friction. The experiment has a similar structure to the experiment ISMIP-HOM $B$. We set $\Omega:=(0,5000)\times (0,1000)$, $\Gamma_d:=\{(0,y);\, y\in [0,1000]\}\cup \{(5000,y);\, y\in [0,1000]\}$, $\Gamma_a:=\{(x,1000);\, x\in (0,5000)\}$, and $\Gamma_b:=\{(x,0);\, x\in (0,5000)\}$. We set $B$, $\rho$, $g$, and $p$ as in the experiment ISMIP-HOM $B$. We set $s:=p$. Again, we rotate the gravity instead of the domain.

We test our experiment with $\tau \in \{10^3,10^7\}$ \textup{Pa a}$^{p-1}$\textup{m}$^{1-p}$. As prescribed by the problem, we have zero velocities at the left and right boundary.
\begin{figure}
	\includegraphics[scale=0.26]{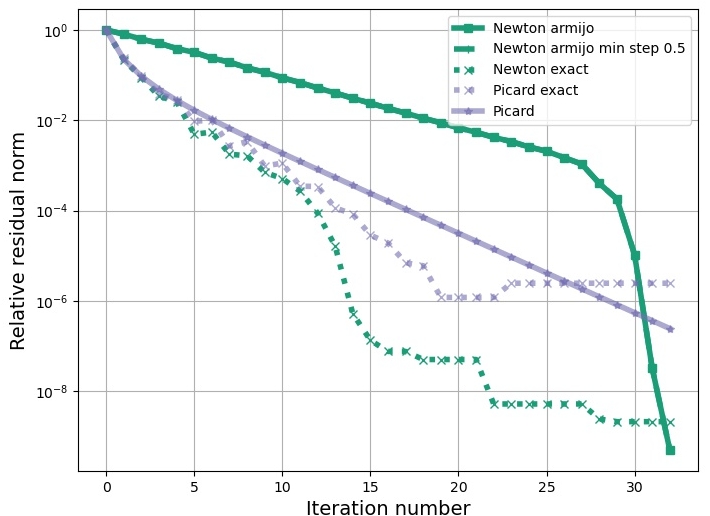}
	\includegraphics[scale=0.26]{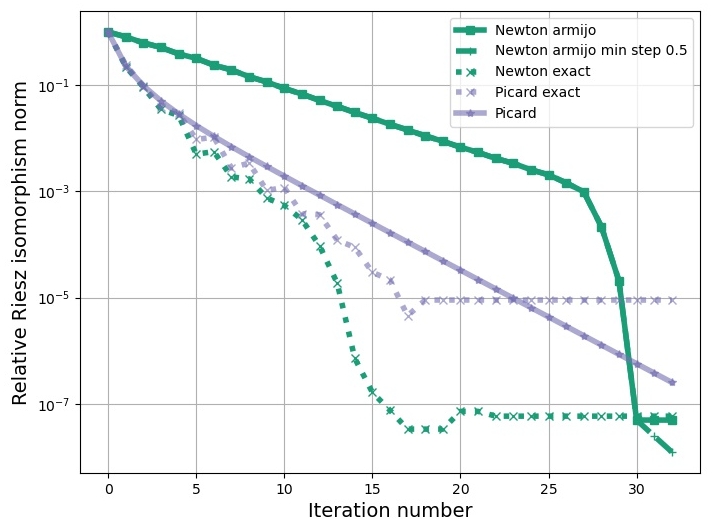}
	\caption{The friction coefficient is $\tau:=10^7$. Left: The initial guess is the same as in experiment ISMIP-HOM $B$. Right: The initial guess has the additional term $\int_{\Gamma_b}\tau \boldsymbol{v_0}\cdot \boldsymbol{\phi}\, ds$.}
	\label{ConvergenceFriction}
\end{figure}
In Fig. $\ref{ConvergenceFriction}$, we see the norms of the Riesz isomorphism for the high friction coefficient with two different initial guesses. On the left, we have the same initial guess as in the ISMIP-HOM experiments. On the right, we added the term $\int_{\Gamma_b} \tau \boldsymbol{v_0}\cdot \boldsymbol{\phi}\, ds$ to the variational formulation of the initial Stokes problem.

Interestingly, Newton's method with Armijo step sizes performs worse than all other algorithms for a high friction coefficient. It converges linearly and needs many iterations until it reaches quadratic convergence. The other initial guess does not resolve this issue. Newton's method with approximately exact step sizes reaches quadratic convergence much earlier. Also, the Picard iteration with approximately exact step sizes is good.
\begin{figure}
	\includegraphics[scale=0.26]{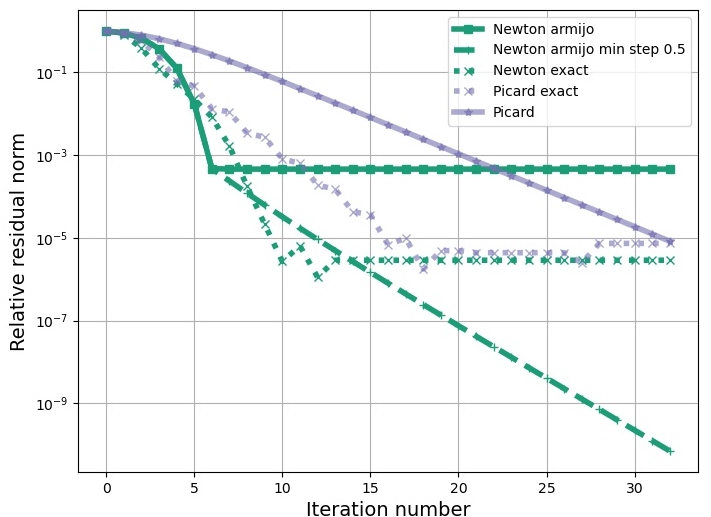}
	\includegraphics[scale=0.26]{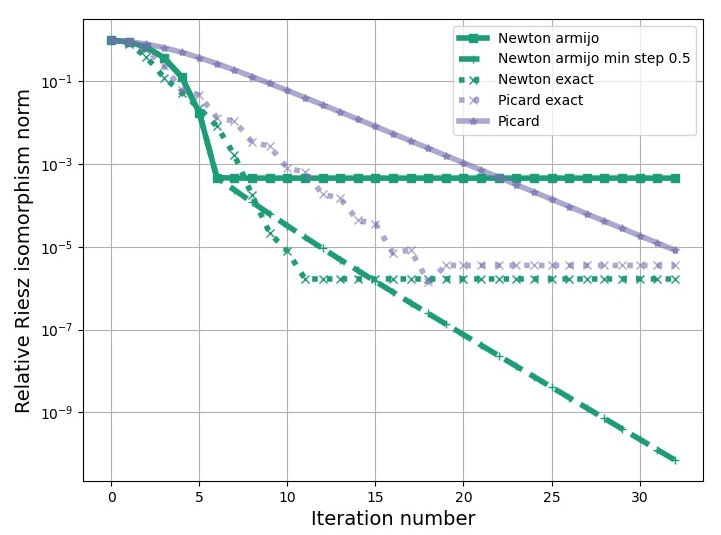}
	\caption{The friction coefficient is $\tau:=10^3$. Left: The initial guess is the same as in experiment ISMIP-HOM $B$. Right: The initial guess has the additional term $\int_{\Gamma_b}\tau \boldsymbol{v_0}\cdot \boldsymbol{\phi}\, ds$.}
	\label{ResidualNorms}
\end{figure}
For a low friction coefficient, all algorithms behave similarly regarding convergence to experiment ISMIP-HOM $B$, see Fig. $\ref{ResidualNorms}$, independent of the initial guess.
\section{Conclusion}\label{Conclusion}In conclusion, it is possible to obtain a global convergent Newton method with step size control with a convex functional. Moreover, the approximation of exact step sizes is also good. For a nonlinear sliding boundary, the convergence rate of Newton's method depends on the friction coefficient. Overall, the approximately exact step sizes seem slightly better because the convergence rate was in all experiments good. However, the convergence rate depends on the size of $\delta$. More experiments should be done in three dimensions, with physically motivated sliding coefficients, and with different values of $\delta$. Furthermore, an implementation in ice sheet models could be interesting.
\section*{Code availability}
An archived version is available at https://doi.org/10.5281/zenodo.10470996.
\section*{Acknowledgments}

I thank my supervisor Prof. Thomas Slawig from Kiel University, for many helpful discussions about mathematical and glaciological problems. Moreover, I thank Prof. Angelika Humbert and Dr. Thomas Kleiner from Alfred-Wegener-Institut in Bremerhaven, Dr. Martin R\"{u}ckamp from the Bavarian Academy of Sciences and Humanities, and Prof. Andreas Rademacher from the University of Bremen for helpful discussions. Finally, I thank the reviewers for their helpful feedback.

This version of the article has been
accepted for publication, after peer review but is not the Version of Record and does not
reflect post-acceptance improvements, or any corrections. The Version of Record is available online at: https://doi.org/10.1007/s11075-024-01941-6
	\bibliographystyle{fbs}

\begin{thebibliography}{10}

\bibitem{Fischler2022}
Fischler, Y., R\"uckamp, M., Bischof, C., Aizinger, V., Morlighem, M., and
  Humbert, A.: A scalability study of the Ice-sheet and Sea-level System Model
  (ISSM, version 4.18).
\newblock \emph{Geoscientific Model Development} 15, 3753--3771 (2022)

\bibitem{Chen2013}
Chen, Q., Gunzburger, M., and Perego, M.: {W}ell-{P}osedness {R}esults for a
  {N}onlinear {S}tokes {P}roblem {A}rising in {G}laciology.
\newblock \emph{{SIAM} Journal on Mathematical Analysis} 45, 2710--2733 (2013)

\bibitem{Jouvet2011}
Jouvet, G. and Rappaz, J.: Analysis and Finite Element Approximation of a
  Nonlinear Stationary Stokes Problem Arising in Glaciology.
\newblock \emph{Advances in Numerical Analysis} 2011, 1--24 (2011)

\bibitem{Hirn2013}
Hirn, A.: Finite element approximation of singular power-law systems.
\newblock \emph{Mathematics of Computation} 82, 1247--1268 (2013)

\bibitem{Diego2023}
de~Diego, G.~G., Farrell, P.~E., and Hewitt, I.~J.: On the Finite Element
  Approximation of a Semicoercive Stokes Variational Inequality Arising in
  Glaciology.
\newblock \emph{{SIAM} Journal on Numerical Analysis} 61, 1--25 (2023)

\bibitem{Fowler2021}
Fowler, A. and Ng, F., eds.: \emph{{G}laciers and {I}ce {S}heets in the
  {C}limate {S}ystem}.
\newblock Springer International Publishing, Cham (2021)

\bibitem{Casas1993}
Casas, E. and Fernandez, L.: {{D}istributed control of systems governed by a
  general class of quasilinear elliptic equations}.
\newblock \emph{J of {D}iff {E}quations} 104, 20--47 (1993)

\bibitem{Arada2012a}
Arada, N.: Distributed control for multistate modified Navier-Stokes equations.
\newblock \emph{{ESAIM}: Control, Optimisation and Calculus of Variations} 19,
  219--238 (2012)

\bibitem{Arada2012}
Arada, N.: Optimal Control of Shear-Thinning Fluids.
\newblock \emph{{SIAM} Journal on Control and Optimization} 50, 2515--2542
  (2012)

\bibitem{Fraters2019}
Fraters, M. R.~T., Bangerth, W., Thieulot, C., Glerum, A.~C., and Spakman, W.:
  {E}fficient and practical {N}ewton solvers for non-linear {S}tokes systems in
  geodynamic problems.
\newblock \emph{Geophysical Journal International} 218, 873--894 (2019)

\bibitem{Allen2017}
Allen, J., Leibs, C., Manteuffel, T., and Rajaram, H.: A Fluidity-Based
  First-Order System Least-Squares Method for Ice Sheets.
\newblock \emph{{SIAM} Journal on Scientific Computing} 39, B352--B374 (2017)

\bibitem{Bons2018}
Bons, P.~D., Kleiner, T., Llorens, M.-G., Prior, D.~J., Sachau, T., Weikusat,
  I., and Jansen, D.: {G}reenland {I}ce {S}heet: {H}igher {N}onlinearity of
  {I}ce {F}low {S}ignificantly {R}educes {E}stimated {B}asal {M}otion.
\newblock \emph{Geophysical Research Letters} 45, 6542--6548 (2018)

\bibitem{MacGregor2016}
MacGregor, J.~A., Fahnestock, M.~A., Catania, G.~A., Aschwanden, A., Clow,
  G.~D., Colgan, W.~T., Gogineni, S.~P., Morlighem, M., Nowicki, S. M.~J.,
  Paden, J.~D., Price, S.~F., and Seroussi, H.: A synthesis of the basal
  thermal state of the Greenland Ice Sheet.
\newblock \emph{Journal of Geophysical Research: Earth Surface} 121, 1328--1350
  (2016)

\bibitem{Chipot2002}
Chipot, M.: \emph{$\ell$ Goes to Plus Infinity}.
\newblock Birkhäuser Basel, Basel (2002)

\bibitem{Amrouche1994}
Amrouche, C.~S. and Girault, V.: Decomposition of vector spaces and application
  to the Stokes problem in arbitrary dimension.
\newblock \emph{Czechoslovak Mathematical Journal} 44, 109--140 (1994)

\bibitem{Girault1986}
Girault, V. and Raviart, P.-A.: \emph{{Finite Element Methods for
  {N}avier-{S}tokes Equations}}.
\newblock Springer Series in Computational Mathematics 5, New York (1986)

\bibitem{Berselli2017}
Berselli, L.~C. and R{\r{u}}{\v{z}}i{\v{c}}ka, M.: Global regularity properties
  of steady shear thinning flows.
\newblock \emph{Journal of Mathematical Analysis and Applications} 450,
  839--871 (2017)

\bibitem{Browder1963}
Browder, F.~E.: {Nonlinear elliptic boundary value problems}.
\newblock \emph{Bulletin of the American Mathematical Society} 69, 862 -- 874
  (1963)

\bibitem{Diening2007}
Diening, L., Ebmeyer, C., and Růžička, M.: Optimal Convergence for the
  Implicit Space-Time Discretization of Parabolic Systems with p-Structure.
\newblock \emph{SIAM Journal on Numerical Analysis} 45, 457--472 (2007)

\bibitem{Hinze2009}
Hinze, M., Pinnau, R., Ulbrich, M., and Ulbrich, S.: \emph{{O}ptimization with
  {PDE} {C}onstraints}.
\newblock Springer Netherlands, Dordrecht (2009)

\bibitem{Polyak2007}
Polyak, R.~A.: Regularized Newton method for unconstrained convex optimization.
\newblock \emph{Mathematical Programming} 120, 125--145 (2007)

\bibitem{Pattyn2008}
Pattyn, F., Perichon, L., Aschwanden, A., Breuer, B., de~Smedt, B.,
  Gagliardini, O., Gudmundsson, G.~H., Hindmarsh, R. C.~A., Hubbard, A.,
  Johnson, J.~V., Kleiner, T., Konovalov, Y., Martin, C., Payne, A.~J.,
  Pollard, D., Price, S., R\"uckamp, M., Saito, F., Sou\v{c}ek, O., Sugiyama,
  S., and Zwinger, T.: {B}enchmark experiments for higher-order and
  full-{S}tokes ice sheet models ({ISMIP}–{HOM}).
\newblock \emph{The Cryosphere} 2, 95--108 (2008)

\bibitem{Logg2012}
Logg, A., Mardal, K.-A., and Wells, G., eds.: \emph{Automated Solution of
  Differential Equations by the Finite Element Method}.
\newblock Springer Berlin Heidelberg, Berlin (2012)

\bibitem{Colinge1999}
Colinge, J. and Rappaz, J.: A strongly nonlinear problem arising in glaciology.
\newblock \emph{{ESAIM}: Mathematical Modelling and Numerical Analysis} 33,
  395--406 (1999)

\bibitem{Larour2012}
Larour, E., Seroussi, H., Morlighem, M., and Rignot, E.: Continental scale,
  high order, high spatial resolution, ice sheet modeling using the Ice Sheet
  System Model ({ISSM}).
\newblock \emph{Journal of Geophysical Research: Earth Surface} 117, n/a--n/a
  (2012)

\bibitem{Isaac2015}
Isaac, T., Stadler, G., and Ghattas, O.: Solution of Nonlinear Stokes Equations
  Discretized By High-Order Finite Elements on Nonconforming and Anisotropic
  Meshes, with Application to Ice Sheet Dynamics.
\newblock \emph{SIAM Journal on Scientific Computing} 37, B804--B833 (2015)

\bibitem{Brown2013}
Brown, J., Smith, B., and Ahmadia, A.: Achieving Textbook Multigrid Efficiency
  for Hydrostatic Ice Sheet Flow.
\newblock \emph{SIAM Journal on Scientific Computing} 35, B359--B375 (2013)

\end{thebibliography}

\end{document}